\documentclass[11pt,letterpaper]{amsart}
\usepackage[foot]{amsaddr}

\usepackage{mathtools}
\usepackage[T1]{fontenc}
\usepackage[latin9]{inputenc}
\usepackage{geometry}
\usepackage{wrapfig}
\usepackage{multicol}
\usepackage{graphicx}
\usepackage{soul}
\usepackage{xcolor}
\usepackage{amssymb}
\usepackage{placeins}
\usepackage{cite}
\usepackage{caption}
\usepackage{enumerate}
\usepackage{afterpage}
\usepackage{enumitem}
\usepackage{bmpsize}
\usepackage{hyperref}
\usepackage{tabu}
\usepackage{enumitem}
\numberwithin{equation}{section}
\usepackage{stmaryrd}
\usepackage{tikz}
\usetikzlibrary{matrix,graphs,arrows,positioning,calc,decorations.markings,shapes.symbols}

\makeatletter
\renewcommand{\email}[2][]{%
  \ifx\emails\@empty\relax\else{\g@addto@macro\emails{,\space}}\fi%
  \@ifnotempty{#1}{\g@addto@macro\emails{\textrm{(#1)}\space}}%
  \g@addto@macro\emails{#2}%
}
\makeatother

\newtheorem{theorem}{Theorem}[section]
\newtheorem{lemma}[theorem]{Lemma}
\newtheorem{proposition}[theorem]{Proposition}

\newtheorem{corollary}[theorem]{Corollary}
{ \theoremstyle{definition}
}
{ \theoremstyle{remark}
\newtheorem{remark}[theorem]{Remark}}

\def\i{ \mathsf{i}}

\newcommand{\R}{\mathbb{R}}

\renewcommand{\P}{\mathbb{P}}

\usepackage{mathtools}

\def\me{ \mu_{\mathsf{eq}}}

\title{A global large deviation principle for discrete $\beta$-ensembles}

\author[E. Dimitrov]{Evgeni Dimitrov}
\address{E. Dimitrov, Department of Mathematics, Rm. 517, Columbia University, New York, NY 10027} \email{esd2138@columbia.edu}

\author[H. Zhang]{Hengzhi Zhang}
\address{H. Zhang, Department of Mathematics, Columbia University, New York, NY 10027} \email{hz2663@columbia.edu}
\begin{document}

\maketitle

\begin{abstract}
We consider discrete $\beta$-ensembles, as introduced by Borodin, Gorin and Guionnet in (Publications math{\' e}matiques de l'IH{\' E}S 125, 1-78, 2017). Under general assumptions, we establish a  large deviation principle for the empirical (or spectral) measures corresponding to these models. Our results apply in the cases when the potential of the model depends on the number of particles, and/or has slow growth near infinity, leading to an equilibrium measure with infinite support.
\end{abstract}

\tableofcontents

%-------------------------------------------------------------------------------------------------------------------------------------------------------------------------------------------------
% Section 1
%
%-------------------------------------------------------------------------------------------------------------------------------------------------------------------------------------------------
\section{Introduction and main results}\label{Section1}
Over the last few decades, several authors (e.g. \cite{agz}, \cite{BG97}, \cite{fe},\cite{Hardy}, \cite{HP00}) have established global large deviation principles (LDPs) for the spectral measures of classical random matrix models such as the Gaussian Orthogonal, Unitary and Symplectic ensembles. Their analysis is based on the description of the joint law of the eigenvalues as a {\em continuous log-gas} (sometimes called {\em Coulomb gas}). More recently, \cite{BGG} proposed a discrete analogue of continuous log-gases, called {\em discrete $\beta$-ensembles} and the goal of this paper is to prove an LDP for the spectral measures corresponding to these models. 

In Section \ref{Section1.1} we recall the definition of the continuous log-gas on $\mathbb{R}$ and some of the results regarding its global large deviations. In Section \ref{Section1.2} we formulate our discrete setup and state our main results. 

%-------------------------------------------------------------------------------------------------------------------------------------------------------------------------------------------------
% Section 1.1
%
%-------------------------------------------------------------------------------------------------------------------------------------------------------------------------------------------------
\subsection{Continuous setting}\label{Section1.1}
Let $\Delta$ be a closed interval in $\mathbb{R}$, $N \in \mathbb{N}$, $\beta > 0$ and $V: \Delta \rightarrow \mathbb{R}$ be a continuous function. A {\em continuous log-gas} is a probability distribution on $\Delta^N$, whose density is 
\begin{equation}\label{clg}
\rho(x_1, \dots, x_N) = \frac{1}{Z_N} {\bf 1}\{x_1 > x_2 >  \cdots > x_N\} \prod_{1\le i<j\le N}(x_{i}-x_j)^{\beta}\prod_{i=1}^{N}\exp\left(-\frac{\beta N}{2} V(x_i) \right),
\end{equation}	
and $Z_N$ is a normalization constant. If $\Delta$ is compact the above measure is well-defined, and if $\Delta$ is unbounded one needs to assume that $V(x)$ grows sufficiently fast as $|x| \rightarrow \infty$ so that
\begin{equation}\label{clgPF}
Z_N := \int_{\Delta^N} {\bf 1} \{ x_1 > x_2 >  \cdots > x_N \} \prod_{1\le i<j\le N}(x_{i}-x_j)^{\beta}\prod_{i=1}^{N}\exp\left(-\frac{\beta N}{2} V(x_i)\right) dx_i < \infty.
\end{equation}	
The quantity $Z_N$ is usually referred to as the {\em partition function} of the model, the parameter $\beta$ is called the {\em inverse temperature} and $V(x)$ is called the {\em potential}.

The study of continuous log-gases for general $\beta > 0$ and potentials $V(x)$ is a rich subject with many important connections to different branches of mathematics. For example, when $\Delta = \mathbb{R}$, $V(x) = x^2$ and $\beta=1,2,$ and $4$, (\ref{clg}) is the joint density of the (ordered) eigenvalues of random matrices from the Gaussian Orthogonal, Unitary and Symplectic ensembles \cite{agz}. We refer the interested reader to the monographs \cite{agz,forrester,mehta} for additional background and a textbook treatment of the classical results on continuous log-gases.

A general question of interest asks about the limiting global distribution of the $x_i$'s as $N \rightarrow \infty$, i.e. about the convergence of the {\em empirical measures}
\begin{equation}\label{S1EM}
\mu_N = \frac{1}{N} \sum_{i = 1}^N \delta_{x_i},
\end{equation}
with $(x_1,\dots, x_N)$ distributed according to (\ref{clg}). Note that $\mu_N$ are random variables that take value in the space $\mathcal{M}(\Delta)$ of probability measures on $\Delta$, which we equip with the usual weak topology. It is known under general assumptions on the potential $V$ that the sequence $\{ \mu_N\}_{N \geq 1}$ converges almost surely to a deterministic probability measure $\me$, called the {\em equilibrium measure}, which is characterized as the unique minimizer of the weighted energy integral 
\begin{equation}\label{S1EF}
\begin{split}
&E_V(\mu) = \iint_{\mathbb{R}^2} k_V(x,y) \mu(dx) \mu(dx), \mbox{ over } \mu \in \mathcal{M}(\Delta), \mbox{ where } \\
& k_V(x,y) = \log |x - y|^{-1} + \frac{1}{2} V(x) + \frac{1}{2} V(y).
\end{split}
\end{equation} 
In fact, a stronger statement in this direction says that $\{ \mu_N\}_{N \geq 1}$ satisfy a large deviation principle on $\mathcal{M}(\Delta)$ with speed $N^2$ and good rate function $I_V(\mu) = (\beta/2) \cdot \left [E_V(\mu) - E_V(\me) \right]$. This result was first established by Ben Arous and Guionnet when $\Delta= \mathbb{R}$, $\beta = 2$ and $V(x) = x^2/2$ in \cite{BG97}. For the case of a general potential $V(x)$, which grows faster than a linear multiple of $\log|x|$ near $\pm \infty$, proofs of the LDP can be found in increasing generality as \cite[Theorem 5.4.3]{HP00},\cite[Theorem 2.6.1]{agz} and \cite[Theorem 1.1]{Hardy}. The LDP of the measures in (\ref{clg}) has also been established when the potential $V(x)$ in (\ref{clg}) is allowed to vary with $N$ in such a way that 
\begin{equation}\label{FeGrowth}
V_N(x) \geq (1 + \xi) \log (1 + x^2) \mbox{ for all $|x| \geq T$ and $N \geq 1$,} 
\end{equation}
where $\xi, T > 0$ are fixed, the functions $V_N$ are continuous and converge uniformly over compact sets of $\Delta$ to a function $V$ -- see \cite[Theorem 2.1]{fe}. 

The above few paragraphs aimed to give a brief overview of the global large deviation problem for continuous log-gases and summarize some of the main results that are available. We next turn to the discrete setup we investigate in the present paper.

%-------------------------------------------------------------------------------------------------------------------------------------------------------------------------------------------------
% Section 1.2
%
%-------------------------------------------------------------------------------------------------------------------------------------------------------------------------------------------------
\subsection{Discrete setting and main results}\label{Section1.2} In this article we consider a discrete analogue of (\ref{clg}), which was introduced in \cite{BGG}. To define the model we begin with some necessary definitions and notation. Let $\theta > 0, N \in \mathbb{N}$ and $a_N \in \mathbb{Z} \cup \{ - \infty\}$, $b_N \in \mathbb{Z} \cup \{ \infty \}$ with $a_N \leq b_N$. We set
\begin{equation}\label{GenState}
\begin{split}
&\mathbb{Y}_N(a_N, b_N) = \{  (\lambda_1, \dots, \lambda_N)\in \mathbb{Z}^N:   a_N \leq \lambda_N \leq \cdots \leq \lambda_1 \leq b_N  \}, \\
&\mathbb{W}^{\theta}_{N}(a_N, b_N)  = \{ (\ell_1, \dots, \ell_N):  \ell_i = \lambda_i + (N - i)\cdot\theta, \mbox{ with } (\lambda_1, \dots, \lambda_N) \in \mathbb{Y}_N(a_N, b_N) \}.
\end{split}
\end{equation}
We interpret $\ell_i$'s as locations of particles. If $\theta = 1$ then all particles live on the integer lattice, while for general $\theta$ the particle of index $i$ lives on the shifted lattice $\mathbb{Z} +(N - i) \cdot \theta$.

We define a probability measure $\mathbb{P}_N^{\theta}$ on $\mathbb{W}^{\theta}_{N}(a_N, b_N)$ through
\begin{equation}\label{PDef}
\mathbb{P}^{\theta}_N(\ell_1, \dots, \ell_N) := \frac1{Z_N} \prod_{1 \leq i < j \leq N} Q_{\theta}(\ell_i-\ell_j)  \prod_{i = 1}^N e^{- \theta N V_N(\ell_i/N)}, \hspace{5mm}  Q_{\theta}(x):=\frac{\Gamma(x + 1)\Gamma(x + \theta)}{\Gamma(x)\Gamma(x +1-\theta)}.
\end{equation}
Here $\Gamma$ is the Euler gamma function, $V_N$ is a continuous functions on $\mathbb{R}$ and
\begin{equation}\label{DBEPF}
Z_N:= \sum_{ (\ell_1, \dots, \ell_N) \in\mathbb{W}^{\theta}_{N}(a_N, b_N)} \prod_{1 \leq i < j \leq N} Q_{\theta}(\ell_i-\ell_j)  \prod_{i = 1}^N e^{- \theta N V_N(\ell_i/N)}
\end{equation}
is a normalization constant (also called the partition function). If $b_N - a_N < \infty$ then $Z_N < \infty$ and (\ref{PDef}) is a well-defined probability measure. If $b_N - a_N = \infty$ then one needs to assume that $V_N$ grows sufficiently fast as $|x| \rightarrow \infty$ to ensure that $Z_N < \infty$. We show in Lemma \ref{WellDef} that as long as 
\begin{equation}\label{VNgrowth}
\liminf_{|x| \rightarrow \infty} N \theta V_N(x)- ( \theta'_N + (N-1) \theta) \log (1 +x^2) > -\infty \mbox{ for some $\theta'_N > 1/2$},
\end{equation}
we have that $Z_N < \infty$ and hence (\ref{PDef}) is a well-defined probability measure.

The measures in (\ref{PDef}) are called {\em discrete $\beta$-ensembles} and were introduced in \cite{BGG} as discrete analogues of (\ref{clg}) and extensively studied. To see why one might consider (\ref{PDef}) as a discrete version of (\ref{clg}), note that $Q_{\theta}(\ell_i-\ell_j) \sim (\ell_i-\ell_j)^{2\theta}$ as $\ell_i - \ell_j \rightarrow \infty$ (see Lemma \ref{InterApprox}), which agrees with \eqref{clg} for $\beta=2\theta$. 

It is worth mentioning that there are other discrete analogues of (\ref{clg}); for example, one can consider the following measure on $\mathbb{W}^{1}_{N}(a_N, b_N)  $ as in (\ref{GenState})
\begin{equation}\label{Coulomb}
\begin{split}
& \mathbb{P}(\ell_1, \dots, \ell_N)  \propto \prod_{1\leq i<j \leq N} |\ell_i-\ell_j|^\beta \prod_{i = 1}^N e^{- \theta N V_N(\ell_i/N)}.
\end{split}
\end{equation}
When $\theta = 2\beta = 1$ the functional equation $\Gamma(z+1) = z \Gamma(z)$ gives $Q_{\theta}(x) = x^2$ so that the measures in (\ref{PDef}) and (\ref{Coulomb}) are the same. For general $\theta = 2\beta$, the measures in (\ref{PDef}) and (\ref{Coulomb}) are {\em different}, since for the former $\ell_i \in \mathbb{Z} +(N - i) \cdot \theta$, while for the latter $\ell_i \in \mathbb{Z}$ for $i = 1, \dots, N$. While both (\ref{PDef}) and (\ref{Coulomb}) are reasonable discretizations of (\ref{clg}), there is a much higher interest in the former coming from connections to integrable probability; specifically, uniform random tilings, $(z,w)$-measures and Jack measures --- see \cite[Section 1]{BGG} for more details. We also mention that (\ref{Coulomb}) appears to lack the integrability that is present in (\ref{PDef}). In particular, while both global and edge fluctuations have been successfully obtained for (\ref{PDef}) in \cite{BGG} and \cite{huang}, respectively, neither has been established for (\ref{Coulomb}), except when $\theta = 1$.\\

Similarly to the continuous setting, we are interested in obtaining a large deviation principle for the empirical measures
\begin{equation}\label{S1DEM}
\mu_N = \frac{1}{N} \sum_{i = 1}^N \delta_{\ell_i/N},
\end{equation}
as $N \rightarrow \infty$. As before, we view $\mu_N$ as random variables taking values in $\mathcal{M}(\mathbb{R}^n)$ with $n = 1$ (here $\mathcal{M}(\mathbb{R}^n)$ is the space of probability measures on $\mathbb{R}^n$, equipped with the weak topology). We mention that the weak topology on $\mathcal{M}(\mathbb{R}^n)$ is compatible with the L{\'e}vy metric $d_n$ defined for two measures $\mu, \nu \in \mathcal{M}(\mathbb{R}^n)$ by
\begin{equation}\label{LevyMetric}
d_n(\mu,\nu) = \inf \{\delta: \mu(F) \leq \nu (F^{\delta}) + \delta \hspace{2mm} \forall F \subseteq \mathbb{R}^n \mbox{ closed}\}, \mbox{ where }
\end{equation}
$F^{\delta} = \{x \in \mathbb{R}^n: \inf_{y \in F} \|x - y\|_n \leq \delta \}$ and $\| \cdot \|_n$ is the Euclidean distance on $\mathbb{R}^n$. Also $(\mathcal{M}( \mathbb{R}^n), d_n)$ is a Polish space, \cite[Theorem C.8]{agz}.

In order to formulate our large deviations theorem we require some additional notation that we now present. If $\Delta \subseteq \mathbb{R}$ is a closed interval, we let $\mathcal{M}(\Delta)$ denote the subset of $\mathcal{M}(\mathbb{R})$ consisting of probability measures $\mu$, whose support $\mathsf{Supp}(\mu)$ is contained in $\Delta$. If $\lambda(\Delta) \geq \theta$ (here $\lambda$ denotes the Lebesgue measure on $\mathbb{R}$) we let $\mathcal{M}_{\theta}(\Delta)$ denote the subset of $\mathcal{M}(\mathbb{R})$ consisting of probability measures $\mu$ that are absolutely continuous with respect to $\lambda$, are supported in $\Delta$ and have density that is bounded by $\theta^{-1}$. The assumption that $\lambda(\Delta) \geq \theta$  ensures that $\mathcal{M}_{\theta}(\Delta)$  is non-empty. The next statement summarizes the relevant facts we require about the function $E_V(\mu)$ from (\ref{S1EF}) when restricted to the set $\mathcal{M}_{\theta}(\Delta)$.

\begin{theorem}\label{ThmFunct} Let $\theta > 0$, $\Delta \subseteq \mathbb{R}$ be a closed interval such that $\lambda(\Delta) \geq \theta$ and $V: \mathbb{R} \rightarrow \mathbb{R}$ be a continuous function satisfying 
\begin{equation}\label{S1VGrowth}
\liminf_{|x| \rightarrow \infty} V(x) - \log(1 + x^2) > -\infty.
\end{equation}
Assume also that $k_V(x,y)$ and $E_V(\mu)$ are as in  (\ref{S1EF}). Then the following statements hold.
\begin{enumerate}
\item For each $\mu \in \mathcal{M}(\mathbb{R})$ the integral $E_V(\mu)$ is well-defined and $E_{V}(\mu) \in (-\infty, \infty]$. 
\item $F_V^{\theta} := \inf_{\mu \in \mathcal{M}_{\theta}(\Delta)} E_V(\mu)$ is finite and there is a unique $\me^{\theta} \in \mathcal{M}_{\theta}(\Delta)$ with $E_V(\me^{\theta}) = F^{\theta}_V$.
\item The function
\begin{equation}\label{S1Rate}
I_V^{\theta}(\mu):= \begin{cases} \theta ( E_V(\mu)  - F_V^{\theta}) &\mbox{ for $\mu \in \mathcal{M}_{\theta}(\Delta)$} \\ \infty & \mbox{ for $\mu \in \mathcal{M}(\mathbb{R}) \setminus  \mathcal{M}_{\theta}(\Delta)$ } \end{cases}
\end{equation}
is a good rate function (GRF) on $\mathcal{M}(\mathbb{R})$.
\item If $\mu \in \mathcal{M}_{\theta}(\Delta)$, $E_V(\mu) < \infty$ and there is a constant $c \in \mathbb{R}$ such that 
\begin{equation*}
\begin{split}
&\int_\R \left(\log|x-y|^{-1}+\frac{1}{2}\log( 1+x^2) \right)\mu(dx) + \frac{1}{2} V(y) \geq c \mbox{, for Lebesgue a.e. $y \in \mathsf{Supp}( \theta^{-1} \lambda - \mu)$} \cap \Delta, \\
&\int_\R \left(\log|x-y|^{-1}+\frac{1}{2}\log( 1+x^2) \right)\mu(dx) + \frac{1}{2} V(y) \leq c \mbox{, for Lebesgue a.e. $y \in \mathsf{Supp}(\mu)$}, 
\end{split}
\end{equation*}
then $\mu = \me^{\theta}$. Here $\mathsf{Supp}$ denotes the support of a measure and $\theta^{-1} \lambda$ denotes the rescaled by $\theta^{-1}$ Lebesgue measure on $\mathbb{R}$. 
\item If $V(x)$ satisfies the stronger, compared to (\ref{S1VGrowth}), growth condition 
\begin{equation}\label{S1VGrowthStrong}
\lim_{|x| \rightarrow \infty} V(x) - \log(1 + x^2) = \infty,
\end{equation}
then the measure $\me^{\theta}$ from (2) above has compact support.
\end{enumerate}
\end{theorem}
\begin{remark}\label{RemFunct} If $V(x)$ is continuous and satisfies (\ref{S1VGrowth}) then we note from the inequality 
\begin{equation}\label{SQLog}
|x-y| \leq \sqrt{1 + x^2} \sqrt{1+y^2} \mbox{ for all $x,y \in \mathbb{R}$ }
\end{equation}
that $k_V(x,y)$ is lower bounded on $\mathbb{R}$ and so $E_V(\mu)$ is well-defined for every $\mu \in \mathcal{M}(\mathbb{R})$ and $E_V(\mu) \in (-\infty, \infty]$. Also, if $V(x)$ is continuous and satisfies (\ref{S1VGrowthStrong}) or if $\Delta$ is bounded then all the statements in Theorem \ref{ThmFunct} follow from \cite[Theorem 2.1]{ds97} and its proof, with the exception of part (4) where one needs to further assume that $\mu \in \mathcal{M}_{\theta}(\Delta)$ is compactly supported if $\Delta$ is unbounded. The non-trivial (and new) parts of Theorem \ref{ThmFunct} are showing that statements (2), (3) and (4) all hold when $\Delta$ is unbounded and $V(x)$ satisfies the growth condition (\ref{S1VGrowth}), rather than (\ref{S1VGrowthStrong}). We establish these three statements in Section \ref{Section2.3}. Our proof of (2) and (3) is based on adapting the arguments in the proof of \cite[Theorem 1.1]{Hardy} (which deals with the case when $\mathcal{M}_{\theta}(\Delta)$ is replaced with $\mathcal{M}(\Delta)$) and our proof of (4) is based on adapting the arguments in the proof of \cite[Theorem 2.1(d)]{ds97}.
\end{remark}

The function $I_V^{\theta}(\mu)$ from (\ref{S1Rate}) is the rate function in our large deviation principle, see Theorem \ref{ThmMain}. We next explain how we scale the parameters in the definition of $\mathbb{P}_N^{\theta}$ in (\ref{PDef}) as $N \rightarrow \infty$.\\

{\bf \raggedleft Assumption 1.} Let $a \in [-\infty, \infty)$ and $b \in (-\infty, \infty]$ be such that $a \leq b$. We assume that $a_N \in \mathbb{Z} \cup \{ - \infty\}$, $b_N \in \mathbb{Z} \cup \{ \infty \}$ with $a_N \leq b_N$ satisfying $\lim_{N \rightarrow \infty} N^{-1} a_N = a$ and $\lim_{N \rightarrow \infty} N^{-1} b_N = b$. We also denote $\Delta = [a, b + \theta] \cap \mathbb{R}$. \\

{\bf \raggedleft Assumption 2.} We assume that $V: \mathbb{R} \rightarrow \mathbb{R}$ is a continuous function, which satisfies (\ref{S1VGrowth}). We also assume that $V_N(x)$ is a sequence of continuous functions on $\mathbb{R}$ such that for each $N \in \mathbb{N}$ the inequality in (\ref{VNgrowth}) holds. As mentioned earlier, the inequality in (\ref{VNgrowth})  implies that $\mathbb{P}_N^{\theta}$ in (\ref{PDef}) is well-defined for any $a_N \in \mathbb{Z} \cup \{ - \infty\}$, $b_N \in \mathbb{Z} \cup \{ \infty \}$, see Lemma \ref{WellDef}. We further assume that at least one of the following two conditions hold:
\begin{enumerate}[label=(\alph*)]
\item There exist $\epsilon_N \geq 0$ such that $|V_N(x) - V(x)| \leq \epsilon_N $ for all $N \geq 1, x \in \mathbb{R}$ and $\lim_{N \rightarrow \infty} \epsilon_N = 0$.
\item There exist $\xi, T > 0$ such that (\ref{FeGrowth}) holds and for each compact set $K \subseteq \mathbb{R}$ we have $\lim_{N \rightarrow \infty} \sup_{x \in K} |V_N(x) - V(x)| = 0$.
\end{enumerate}
In words, Assumption 2 states that either $V_N$ converge uniformly to $V(x)$ on $\mathbb{R}$, which satisfies the mild growth condition (\ref{S1VGrowth}), or the $V_N$ converge to $V$ uniformly over compact sets, but have the uniform (in $N$) growth condition in (\ref{FeGrowth}). The only other aspect of Assumption 2, is that $V_N$ satisfy (\ref{VNgrowth}) so that the measure $\mathbb{P}_N^{\theta}$ in (\ref{PDef}) is well-defined.

With the above notation in place we can present the main result of the paper.
\begin{theorem}\label{ThmMain} Suppose that $\theta > 0$ and $\mathbb{P}_N^{\theta}$ is a sequence of probability measures on $\mathbb{W}^{\theta}_{N}(a_N, b_N) $ as in (\ref{GenState}) of the form (\ref{PDef}), where $a_N, b_N$ satisfy the conditions in Assumption 1 and $V_N$ satisfy the conditions in Assumption 2. If $\mu_N$ are as in (\ref{S1DEM}) for $(\ell_1, \dots, \ell_N)$ distributed according to $\mathbb{P}_N^{\theta}$, then the sequence of measures in $\mathcal{M}(\mathbb{R})$, given by the laws of $\mu_N$, satisfies an LDP with speed $N^2$ and good rate function $I_V^{\theta}(\mu)$ as in (\ref{S1Rate}), with $\Delta$ as in Assumption 1.
\end{theorem}
Theorem \ref{ThmMain} is proved in Section \ref{Section3.1} and is based on an adaptation of an elegant idea from \cite{Hardy}, which goes back to \cite{Hardy2}. The core of this idea is to map the measures $\mu_N$ in Theorem \ref{ThmMain} to ones on the circle $\mathcal{S} := \{ \vec{x} = (x_1, x_2) \in \mathbb{R}^2: x_1^2 + (x_2 - 1/2)^2 = 1 \}$, using the inverse stereographic projection $T$, that maps the one-point compactification of $\mathbb{C}$ to the Riemann sphere, and restricting it to $\mathbb{R}$. As explained in \cite[Remark 1.5]{Hardy}, the advantage of working with the space of probability measures on $\mathcal{S}$ is that the latter is compact in the weak topology. One consequence of this compactness is that it suffices to prove a weak LDP upper bound for the measures $T_*\mu_N$ (the pushforwards of $\mu_N$ by $T$), and as a result we obtain a strong LDP upper bound for $T_*\mu_N$, and consequently for $\mu_N$. Another consequence of the compactness is that it allows one to circumvent the necessity of establishing an exponential tightness property for $\mu_N$, which is required in classical works on LDPs for continuous and discrete log-gases \cite{agz, BG97, fe, HP00}. In fact, as mentioned in \cite[Remark 1.5]{Hardy}, it is not even clear how to directly establish the exponential tightness of $\mu_N$ under the weaker growth in (\ref{VNgrowth}). 

While the core idea of our work is similar to \cite{Hardy}, there are several new challenges that we face, which come from the discreteness of our models in (\ref{PDef}). For example, when we prove the weak LDP upper bound of $T_*\mu_N$ in Lemma \ref{S3WLDP}, the arguments in \cite{Hardy} only work under Assumption 2(a), and provide the correct rate function only on the set $T_*(\mathcal{M}_{\theta}(\Delta))$, see Section \ref{Section3.2.1}. To complete the proof under Assumption 2(b) as well, and also find the correct rate function on all of probability measures on $\mathcal{S}$, we adapt some of the ideas from \cite{fe,jo}, see Sections \ref{Section3.2.2} and \ref{Section3.2.3}. Obtaining the strong LDP upper bound for $\mu_N$ from Lemma \ref{S3WLDP} is done in Step 1 of the proof of Theorem \ref{ThmMain} in Section \ref{Section3.1}, and essentially uses the argument in \cite{Hardy}, which is inspired by \cite[Theorem 4.1.1]{DZ}. For the LDP lower bound, \cite{Hardy} relies on \cite[Theorem 2.6.1]{agz}, which is not applicable in our case due to the discreteness of our models. Consequently, we need to develop this part of the proof separately, for which we rely on some ideas from discrete log-gases \cite{fe,jo}, as well as a couple of technical lemmas -- Lemmas \ref{ConvL} and \ref{CompactL}.

To summarize, we have attempted to prove Theorem \ref{ThmMain} under the weakest possible conditions, when the intervals of support and the potentials $V_N$ for the measures $\mathbb{P}^{\theta}_N$ in (\ref{PDef}) are allowed to vary with $N$, and we assume as little as possible about them. Under Assumption 2(a) one can adapt the arguments from \cite{Hardy}, and under Assumption 2(b) one can adapt the arguments from \cite{fe,jo} to get the LDP upper bound, but we need to modify both types of arguments to account for the varying nature of our supports $\mathbb{W}_N^{\theta}(a_N,b_N)$ and potentials $V_N$. For the LDP lower bound, we appropriately modify the argument from \cite{Hardy}, which is for continuous log-gases to our discrete setting.

%-------------------------------------------------------------------------------------------------------------------------------------------------------------------------------------------------
% Section 1.3
%
%-------------------------------------------------------------------------------------------------------------------------------------------------------------------------------------------------	
\subsection{Outline} The rest of the article is organized as follows. In Section \ref{Section2.1}, we provide sufficient conditions for the measures $\mathbb{P}_N^{\theta}$ to be well-defined. Section \ref{Section2.2} explains the compactification argument from \cite[Section 2]{Hardy} and in Section \ref{Section2.3} we prove Theorem \ref{ThmFunct}. In Section \ref{Section3.1} we state a certain weak LDP upper bound for the pushforward measures of $\mu_N$ under the map $T$ in Section \ref{Section2.2}, this is Lemma \ref{S3WLDP}, as well as two technical results -- Lemmas \ref{ConvL} and \ref{CompactL}. Within the same section, we prove Theorem \ref{ThmMain} using these three results. Lemma \ref{S3WLDP} is proved in Section \ref{Section3.2}, while Lemmas \ref{ConvL} and \ref{CompactL} are proved in Section \ref{Section3.3}. In Section \ref{Section4} we give two applications of Theorem \ref{ThmMain}. One of them is to certain measures related to Jack symmetric functions, and the other is to certain discrete analogues of the Cauchy ensembles from \cite[Example 1.3]{Hardy}.

%-------------------------------------------------------------------------------------------------------------------------------------------------------------------------------------------------
% Section 1.4
%
%-------------------------------------------------------------------------------------------------------------------------------------------------------------------------------------------------
\subsection*{Acknowledgments} The authors would like to thank Vadim Gorin for useful comments on earlier drafts of the paper. E.D. is partially supported by NSF grant DMS:2054703.

%-------------------------------------------------------------------------------------------------------------------------------------------------------------------------------------------------
% Section 2
%
%-------------------------------------------------------------------------------------------------------------------------------------------------------------------------------------------------
\section{Preliminary results}\label{Section2} In Section \ref{Section2.1} we give sufficient conditions under which the measures in (\ref{PDef}) are well-defined. In Section \ref{Section2.2} we recall the compactification argument from \cite[Section 2]{Hardy} and in Section \ref{Section2.3} we present the proof of Theorem  \ref{ThmFunct}. We continue with the same notation as in Section \ref{Section1.2}.

%-------------------------------------------------------------------------------------------------------------------------------------------------------------------------------------------------
% Section 2.1
%
%-------------------------------------------------------------------------------------------------------------------------------------------------------------------------------------------------
\subsection{Well-posedness of $\mathbb{P}_N^{\theta}$}\label{Section2.1}

We recall the following result from \cite{DD21}.
\begin{lemma}\label{InterApprox}\cite[Lemma 2.14]{DD21} For any $x \geq \theta > 0$ we have
\begin{equation}\label{Sandwich}
Q_\theta(x) = \frac{\Gamma(x + 1)\Gamma(x+ \theta)}{\Gamma(x)\Gamma(x +1-\theta)} \in \left[  x^{2\theta} \cdot \exp(-(1+\theta)^3x^{-1}),  x^{2\theta} \cdot \exp((1+\theta)^3x^{-1}) \right].
\end{equation}
\end{lemma}

\begin{lemma}\label{WellDef} Fix $\theta > 0$, $N \in \mathbb{N}$, $a_N \in \mathbb{Z} \cup \{ - \infty\}$, $b_N \in \mathbb{Z} \cup \{ \infty \}$ with $a_N \leq b_N$, and $\mathbb{W}^{\theta}_{N}(a_N, b_N)$ as in (\ref{GenState}). Suppose $V_N: \mathbb{R} \rightarrow \mathbb{R}$ is continuous and satisfies (\ref{VNgrowth}). For each $\ell= (\ell_1, \dots, \ell_N) \in \mathbb{W}^{\theta}_{N}(a_N, b_N)$ define 
$$W(\ell) = \prod_{1 \leq i < j \leq N} Q_{\theta}(\ell_i-\ell_j)  \prod_{i = 1}^N e^{- \theta N V_N(\ell_i/N)},$$
where $Q_{\theta}$ is as in (\ref{Sandwich}). Then, $W(\ell) > 0$ for all $\ell \in\mathbb{W}^{\theta}_{N}(a_N, b_N)$, and $Z_N : = \sum_{ \ell \in \mathbb{W}^{\theta}_{N}(a_N, b_N)} W(\ell)  \in (0, \infty)$. In particular, (\ref{PDef}) is a well-defined probability measure.
\end{lemma}
\begin{proof}

The positivity of $ W(\ell) $ follows from the positivity of the gamma function on $(0, \infty)$ and the positivity of exponential functions. Thus, we only need to prove that 
\begin{equation}\label{L2R1}
Z_N^{\infty} := \sum_{ \ell \in \mathbb{W}_N^{\theta}(-\infty, \infty)} \prod_{1 \leq i < j \leq N} Q_{\theta}(\ell_i-\ell_j) \cdot  \prod_{i = 1}^N e^{- \theta N V_N(\ell_i/N)} < \infty.
\end{equation}
The continuity of $V_N$ and (\ref{VNgrowth}) imply that we can find $A > 0$ such that for all $x \in \mathbb{R}$
$$- \theta N V_N(x) \leq A - (\theta'_N+(N-1) \theta )\log(1+x^2)$$
Combining the latter with Lemma \ref{InterApprox}, we conclude that 
\begin{equation}\label{ZUB1}
\begin{split}
&Z_N^{\infty} = \sum_{ \ell \in \mathbb{W}_N^{\theta}(-\infty, \infty)} \prod_{1 \leq i < j \leq N} Q_{\theta}(\ell_i-\ell_j) \cdot  \prod_{i = 1}^N e^{- \theta N V_N(\ell_i/N)}    \\
&  \leq e^{AN+(1+\theta)^3B }N^{N(N-1)\theta} \hspace{-2mm} \sum\limits_{\ell\in\mathbb{W}_N^{\theta}(-\infty, \infty) }\prod\limits_{1\leq i<j\leq N}(\ell_i/N- \ell_j/N)^{2\theta}\prod\limits_{i=1}^Ne^{-(\theta'_N+(N-1)\theta)\log(1+(\ell_i/N)^2)},
\end{split}
\end{equation}
where we have set $B : = \sum_{1 \leq i < j \leq N} \frac{1}{(j-i) \theta}$ and used that for $ {\ell} \in \mathbb{W}_N^{\theta}(-\infty, \infty)$ we have
$$\sum_{1 \leq i < j \leq N} \frac{1}{\ell_i - \ell_j} \leq \sum_{1 \leq i < j \leq N} \frac{1}{(j-i) \theta} = B.$$
Combining (\ref{ZUB1}) with (\ref{SQLog}) we get that there exists $C > 0$, depending on $N$ and $\theta$, such that 
$$Z_N^{\infty}  \leq C \cdot \sum\limits_{{\ell}\in\mathbb{W}_N^{\theta}(-\infty, \infty) }\prod\limits_{i=1}^Ne^{-\theta'_N \log(1+(\ell_i/N)^2)} \leq C \cdot \prod_{i = 1}^N \left(\sum_{x \in \mathbb{Z}} \frac{1}{\left(1 + (x + (N-i)\cdot \theta)^2/ N^2 \right)^{\theta'_N}} \right).$$
The last inequality implies (\ref{L2R1}) since $\theta'_N > 1/2$ by assumption.
\end{proof}

%-------------------------------------------------------------------------------------------------------------------------------------------------------------------------------------------------
% Section 2.2
%
%-------------------------------------------------------------------------------------------------------------------------------------------------------------------------------------------------
\subsection{Compactification}\label{Section2.2} In this section, we describe the compactification procedure from \cite[Section 2]{Hardy}, and recall the results from that paper that we require. We mention that the setup in \cite{Hardy} is for $\mathbb{C}$, but can be readily adapted to $\mathbb{R}$, using the usual embedding of $\mathbb{R}$ in $\mathbb{C}$.

Let $\mathcal{S} \subseteq \mathbb{R}^2$ be given by 
$$\mathcal{S} := \{ \vec{x} = (x_1, x_2) \in \mathbb{R}^2: x_1^2 + (x_2 - 1/2)^2 = 1 \}.$$
In words, $\mathcal{S}$ is the circle of radius $1/2$, centered at the point $(0, 1/2)$. We let $T: \mathbb{R} \rightarrow \mathcal{S}$ be 
$$T(x) := \left( \frac{x}{1 + x^2}, \frac{x^2}{1 + x^2} \right),$$
and note that $T$ is a homeomorphism from $\mathbb{R}$ onto $\mathcal{S} \setminus \{ \mathsf{np} \}$, where we write $\mathsf{np} = (0,1)$. If $\Delta \subseteq \mathbb{R}$ is a closed interval, we set 
\begin{equation}\label{S2CompDelta}
\Delta_{\mathcal{S}} = T(\Delta) \cup \{ \mathsf{np}\},
\end{equation}
and note that $\Delta_{\mathcal{S}}$ is a closed subset of $\mathcal{S}$, and hence $\mathbb{R}^2$. For a closed subset $F \subseteq \mathbb{R}^2$ we endow it with the subspace topology, coming from the usual topology on $\mathbb{R}^2$, and corresponding Borel $\sigma$-algebra. In addition, we write $\mathcal{M}(F)$ for the set of probability measures in $\mathcal{M}(\mathbb{R}^2)$, whose support is contained in $F$. We endow $\mathcal{M}(F)$ with the weak topology, which is the same as the subspace topology coming from $\mathcal{M}(\mathbb{R}^2)$. In particular, if $d_2$ denotes the L{\'e}vy metric in (\ref{LevyMetric}) then $(\mathcal{M}(F), d_2)$ is a Polish space, and the metric topology is the same as the weak topology on $\mathcal{M}(F)$.

Given $\mu \in \mathcal{M}(\mathbb{R})$, we let $T_* \mu$ denote the push-forward of $\mu$ under the map $T$, i.e. for each Borel set $A \subseteq \mathcal{S}$ we have 
\begin{equation}\label{S2Push}
T_*\mu (A) := \mu (T^{-1} (A)).
\end{equation}
We now state the first result we require from \cite[Section 2]{Hardy}.
\begin{lemma}\label{S2TDiff}\cite[Lemma 2.1]{Hardy} For each closed interval $\Delta \subseteq \mathbb{R}$ the map $T_*$ is a homeomorphism from $\mathcal{M}(\Delta)$ to $\{ \nu \in \mathcal{M}(\Delta_\mathcal{S}): \nu(\{ \mathsf{np}\}) = 0\}.$
\end{lemma}

Our next task is to reformulate the minimization problem of (\ref{S1EF}), which is defined for measures on $\mathbb{R}$, to one that is defined for measures on $\mathcal{S}$. Below we assume that $V(x)$ is a continuous function on $\mathbb{R}$ that satisfies (\ref{S1VGrowth}). We define the function $\mathcal{V}: \mathcal{S} \rightarrow (-\infty, \infty]$ through
\begin{equation}\label{S2VtoV}
\mathcal{V}(\vec{x}) = \begin{cases} V(y) - \log (1 + y^2) &\mbox{ if $\vec{x} \neq \mathsf{np}$ and $y = T^{-1}(\vec{x})$ } \\ \liminf_{|x| \rightarrow \infty}  V(x) - \log (1 + x^2)&\mbox{ if $\vec{x} = \mathsf{np}$.} \end{cases}
\end{equation}
The growth condition (\ref{S1VGrowth}) ensures that $\mathcal{V}$ is lower semi-continuous and bounded from below on $\mathcal{S}$ so that the function $F_{\mathcal{V}}: \mathcal{S} \times \mathcal{S} \rightarrow (-\infty, \infty]$, given by
\begin{equation}\label{S2Ker}
F_{\mathcal{V}}(\vec{x}, \vec{y}) = \log \| \vec{x} - \vec{y} \|_2^{-1} + \frac{1}{2} \mathcal{V}(\vec{x}) + \frac{1}{2} \mathcal{V}(\vec{y}), \hspace{2mm} \vec{x}, \vec{y} \in \mathcal{S},
\end{equation}
is lower semi-continuous and bounded from below on $\mathcal{S} \times \mathcal{S}$, where we recall that $\|\cdot \|_n$ is the Euclidean distance on $\mathbb{R}^n$. The latter implies that the weighted energy integral 
\begin{equation}\label{S2EF}
E_{\mathcal{V}}(\nu) = \iint_{\mathcal{S}^2} F_{\mathcal{V}}(\vec{x},\vec{y}) \nu(d\vec{x}) \nu(d\vec{y}), \hspace{2mm} \nu \in \mathcal{M}(\mathcal{S})
\end{equation}
is well-defined, and takes values in $(-\infty, \infty]$. In addition, from \cite[Equation (2.9)]{Hardy} we have that if $V(x)$ is continuous and satisfies (\ref{S1VGrowth}), and $E_V$ is as in (\ref{S1EF}), then
\begin{equation}\label{S2EFEq}
E_V(\mu) = E_{\mathcal{V}}(T_*\mu)  \mbox{ for all $\mu \in \mathcal{M}(\mathbb{R})$,} 
\end{equation}
provided that $T_*$ is as in (\ref{S2Push}), and $\mathcal{V}$ is as in (\ref{S2VtoV}). Equation (\ref{S2EFEq}) is the key identity, which allows us to transport the minimization problem in (\ref{S1EF}) over measures in $\mathcal{M}(\mathbb{R})$ to one over measures in $\mathcal{M}(\mathcal{S})$, the latter space being more convenient in view of its compactness.\\

We end this section by formulating a useful proposition, which can be found as \cite[Proposition 2.3]{Hardy}. As our formulation is slightly different, we will also provide the proof of the proposition for completeness. We mention that the core of our proof is the same as that of \cite[Proposition 2.3]{Hardy} and relies on an appropriate application of \cite[Theorem 2.5]{CKL98}. Our main contribution is in providing a more detailed justification of why \cite[Theorem 2.5]{CKL98} is applicable, compared to \cite[Proposition 2.3]{Hardy}.

\begin{proposition}\label{S2GRF} Let $V$ be a continuous function on $\mathbb{R}$ that satisfies (\ref{S1VGrowth}), and let $\mathcal{V}$ be as in (\ref{S2VtoV}). If $E_{\mathcal{V}}$ is as in (\ref{S2EF}), then 
\begin{enumerate}[label=(\alph*)]
\item For each $\alpha \in \mathbb{R}$ the set $\{ \nu \in \mathcal{M}(\mathcal{S}): E_{\mathcal{V}}(\nu) \leq \alpha \}$ is compact in $\mathcal{M}(\mathcal{S})$. 
\item The function $E_{\mathcal{V}}$ is strictly convex on $\mathcal{M}(\mathcal{S})$ in the sense that for any $\mu, \nu \in \mathcal{M}(\mathcal{S})$ and $t \in (0,1)$ we have
\begin{equation}\label{S2Convex}
E_{\mathcal{V}}(t \mu + (1- t) \nu) \leq t E_{\mathcal{V}}(\mu) + (1- t) E_{\mathcal{V}}(\nu), 
\end{equation}
and the inequality in (\ref{S2Convex}) is strict if $\mu \neq \nu$ (here $\infty < \infty$ is allowed).
\end{enumerate}
\end{proposition}
\begin{proof} For clarity, we split the proof into two steps. In the first step we prove the proposition modulo a certain inequality, see (\ref{S2Ineq}), and in the second step we establish (\ref{S2Ineq}). The inequality (\ref{S2Ineq}) will be used to show that \cite[Theorem 2.5]{CKL98} is applicable, and its proof is based on adapting some ideas from the proof of \cite[Lemma 2.6.2]{agz}.\\

{\bf \raggedleft Step 1.} We define for $\mu, \nu \in \mathcal{M}(\mathcal{S})$ the function 
\begin{equation}\label{S2MixedFunct}
I(\mu, \nu) := \iint_{\mathcal{S}^2} \log \| \vec{x} - \vec{y} \|_2^{-1} \mu(d\vec{x}) \nu(d\vec{y}),
\end{equation}
and note that as $ \| \vec{x} - \vec{y} \|_2 \leq 1$ the integrand is in $[0, \infty]$, so that the integral is well-defined. We claim that for any $\mu, \nu \in \mathcal{M}(\mathcal{S})$ we have
\begin{equation}\label{S2Ineq}
2I(\mu,\nu) \leq I(\mu,\mu) + I(\nu,\nu).
\end{equation} 
We will establish (\ref{S2Ineq}) in the second step. Here, we assume its validity and conclude the proof of the proposition.\\

We proceed to prove (a). Since $F_{\mathcal{V}}(\vec{x},\vec{y})$ is lower semi-continuous on $\mathcal{S}$, there exists an increasing sequence $F_{\mathcal{V}}^n(\vec{x},\vec{y})$ of continuous functions, which converge to $F_{\mathcal{V}}(\vec{x},\vec{y})$ from below. From the monotone convergence theorem
$$E_{\mathcal{V}}(\mu)=\sup\limits_n\iint_{\mathcal{S}^2} F_{\mathcal{V}}^n(\vec{x},\vec{y})\mu(\vec{x})d\mu(d\vec{y}).$$ 
We conclude that $E_\mathcal{V}$ is lower semi-continuous on $\mathcal{M}(\mathcal{S})$, and so $\{\nu\in\mathcal{M}(\mathcal{S}):E_\mathcal{V}(\nu)\leq \alpha\}$ is closed. Since $\mathcal{M}(\mathcal{S})$ is compact we conclude the same for $\{\nu\in\mathcal{M}(\mathcal{S}):E_\mathcal{V}(\nu)\leq \alpha\}$.

We next prove part (b). Since $\log \| \vec{x} - \vec{y} \|_2^{-1} \geq 0$ and $\mathcal{V}$ is lower bounded on $\mathcal{S}$ we have
$$E_\mathcal{V}(\nu)=I(\nu,\nu)+\int_{\mathcal{S}} \mathcal{V}(\vec{x}) \nu(d\vec{x}),$$
for each $\nu \in \mathcal{M}(\mathcal{S})$. Consequently, it suffices to prove $I(\nu,\nu)$ is strictly convex on $\mathcal{M}({S})$, i.e. (\ref{S2Convex}) holds with $E_{\mathcal{V}}(\nu)$ replaced with $I(\nu,\nu)$.

If $I(\mu,\mu)=\infty$ or $I(\nu,\nu)=\infty$, then the equation trivially holds (recall that $\infty<\infty$ is allowed). If both $I(\mu,\mu)<\infty$ and $I(\nu,\nu)<\infty$, then from (\ref{S2MixedFunct}) and (\ref{S2Ineq}), linearity of the integral and the inequality $\log \| \vec{x} - \vec{y} \|_2^{-1} \geq 0$ we get
\begin{equation}\label{S2Appl}
I(|\nu-\mu|,|\nu-\mu|)\leq I(\nu+\mu,\nu+\mu) = I(\mu,\mu)+I(\nu,\nu) +2I(\mu,\nu) <\infty.
\end{equation}
By linearity, we have 
$$I(t\mu+(1-t)\nu,t\mu+(1-t)\nu)=tI(\mu,\mu)+(1-t)I(\nu,\nu)-t(1-t)I(\mu-\nu,\mu-\nu).$$ 
Equation (\ref{S2Appl}) shows that \cite[Theorem 2.5]{CKL98} is applicable, and the latter gives $I(\mu-\nu,\mu-\nu)\geq 0$ with equality if and only if $\mu = \nu$. The last two statements conclude the proof of the strict convexity of $I(\nu,\nu)$ and hence part (b).\\

{\bf \raggedleft Step 2.} In this step, we prove (\ref{S2Ineq}).
From \cite[(2.6.19)]{agz} we have
\begin{equation*}
\log\| \vec{x} - \vec{y} \|_2^{-1}=\int_0^\infty \frac{1}{2t}\left (\exp\left (-\frac{1}{2t}\right)-\exp\left(-\frac{\| \vec{x} - \vec{y} \|_2^2}{2t} \right) \right)dt,
\end{equation*}
from which we conclude that  
\begin{equation}\label{NewI}
\begin{split}
I(\mu, \nu) =& \hspace{1mm} \iint_{\mathcal{S}^2} \int_0^{\infty} \frac{1}{2t}\left (\exp\left (-\frac{1}{2t}\right)-\exp\left(-\frac{\| \vec{x} - \vec{y} \|_2^2}{2t} \right) \right) \log \| \vec{x} - \vec{y} \|_2^{-1}dt \mu(d\vec{x}) \nu(d\vec{y}) \\
 = & \hspace{1mm} \int_0^{\infty} \iint_{\mathcal{S}^2} \frac{1}{2t}\left (\exp\left (-\frac{1}{2t}\right)-\exp\left(-\frac{\| \vec{x} - \vec{y} \|_2^2}{2t} \right) \right) \log \| \vec{x} - \vec{y} \|_2^{-1} \mu(d\vec{x}) \nu(d\vec{y})dt.
\end{split}
\end{equation}
In going from the first to the second line we used that $\| \vec{x} - \vec{y} \|_2 \leq 1$, which implies that the integrand is non-negative and the order of integration can be exchanged by Tonelli's theorem.

In view of (\ref{NewI}), we see that to show  (\ref{S2Ineq}) it suffices to prove that for each $t \in (0,\infty)$ we have 
\begin{equation*}
\begin{split}
& 2 \iint_{\mathcal{S}^2}\left (\exp\left (-\frac{1}{2t}\right)-\exp\left(-\frac{\| \vec{x} - \vec{y} \|_2^2}{2t} \right) \right) \log \| \vec{x} - \vec{y} \|_2^{-1} \mu(d\vec{x}) \nu(d\vec{y})   \\
& \leq \iint_{\mathcal{S}^2}\left (\exp\left (-\frac{1}{2t}\right)-\exp\left(-\frac{\| \vec{x} - \vec{y} \|_2^2}{2t} \right) \right) \log \| \vec{x} - \vec{y} \|_2^{-1} \mu(d\vec{x}) \mu(d\vec{y}) \\
& + \iint_{\mathcal{S}^2}\left (\exp\left (-\frac{1}{2t}\right)-\exp\left(-\frac{\| \vec{x} - \vec{y} \|_2^2}{2t} \right) \right) \log \| \vec{x} - \vec{y} \|_2^{-1} \nu(d\vec{x}) \nu(d\vec{y}).
\end{split}
\end{equation*}
At this point all the integrands have finite integrals and we can use linearity and symmetry of the first line in $\mu$ and $\nu$ to reduce the above inequality to
\begin{equation}\label{S2RI1}
0\leq \iint_{\mathcal{S}^2} \exp\left(-\frac{\| \vec{x} - \vec{y} \|_2^2}{2t} \right)(\mu-\nu)(d\vec{x})(\mu-\nu)(d\vec{y}).
\end{equation}
Writing as usual $\vec{x} = (x_1, x_2)$ and $\vec{y} = (y_1, y_2)$, and using the identity 
$$e^{-(x-y)^2/2t} = \sqrt{\frac{t}{2\pi}} \int_{\mathbb{R}} e^{\i (x-y) \lambda} e^{-t\lambda^2/2} d\lambda,$$
which is nothing but the characteristic function of a normal variable with mean $0$ and variance $t^{-1}$, we see that 
\begin{equation*}
\begin{split}
&\iint_{\mathcal{S}^2} \exp\left(-\frac{\| \vec{x} - \vec{y} \|_2^2}{2t} \right)(\mu-\nu)(d\vec{x})(\mu-\nu)(d\vec{y})  \\
& = \iint_{\mathcal{S}^2} \exp\left(-\frac{(x_1 - y_1)^2}{2t} -\frac{(x_2 - y_2)^2}{2t}  \right)(\mu-\nu)(d\vec{x})(\mu-\nu)(d\vec{y}) \\
& = \frac{t}{2\pi}\int_{\R}\int_{\R}e^{-t\lambda_1^2/2}e^{-t\lambda_2^2/2}\int_{\mathcal{S}}e^{\i (x_1 - y_1) \lambda_1} (\mu-\nu)(d\vec{x})\int_{\mathcal{S}}e^{\i (x_2 - y_2) \lambda_2} (\mu-\nu)(d\vec{y})d\lambda_1d\lambda_2  \\
& = \frac{t}{2\pi}\int_{\R}\int_{\R}e^{-t\lambda_1^2/2}e^{-t\lambda_2^2/2}\left |\int_{\mathcal{S}}e^{\i \lambda_1x_1+\i \lambda_2x_2}(\mu-\nu)(d\vec{x}) \right|^2d\lambda_1d\lambda_2.
\end{split}
\end{equation*}
The last equality implies (\ref{S2RI1}) as the last integrand is non-negative. This suffices for the proof.
\end{proof}

%-------------------------------------------------------------------------------------------------------------------------------------------------------------------------------------------------
% Section 2.3
%
%-------------------------------------------------------------------------------------------------------------------------------------------------------------------------------------------------
\subsection{Proof of Theorem \ref{ThmFunct}}\label{Section2.3} As explained in Remark \ref{RemFunct}, we only need to show that parts (2), (3) and (4) all hold when $\Delta$ is unbounded and $V(x)$ is a continuous function that satisfies the growth condition in (\ref{S1VGrowth}). We mention that the arguments below are inspired by the proofs of \cite[Theorem 2.1]{ds97} and \cite[Theorem 1.1]{Hardy}. For clarity we split the proof into two steps.\\

{\bf \raggedleft Step 1.} In this step we prove part (2). From our assumption that $\lambda (\Delta) \geq \theta$, we can find a closed interval $\Delta' = [a', a' + \theta] \subseteq \Delta$. We observe that if $\mu$ has density $\theta^{-1} \cdot {\bf 1}\{x \in [a',a' + \theta]\}$, then $\mu \in\mathcal{M}_\theta(\Delta)$ and $E_V(\mu)<\infty$. This implies that $\mathcal{M}_\theta(\Delta)$ is non-empty, $F_V^\theta<\infty$ and $E_V$ is not identically equal to $\infty$ on $\mathcal{M}_\theta(\Delta)$. 

From Proposition \ref{S2GRF} we know for every $\alpha \in \mathbb{R}$ that $ \left\{ \nu \in \mathcal{M}(\mathcal{S}) : E_{\mathcal{V}}(\nu) \leq \alpha \right\}$ is compact. In addition, $\mathcal{M}(\Delta_{\mathcal{S}})$ is a closed subset of $\mathcal{M}(\mathcal{S})$, since $\Delta_{\mathcal{S}}$ is a closed subset of $\mathcal{S}$. We thus conclude that for every $\alpha \in \mathbb{R} $ the set $ \left\{ \nu \in \mathcal{M}(\Delta_{\mathcal{S}}) : E_{\mathcal{V}}(\nu) \leq \alpha \right\}$ is a compact subset of $\mathcal{M}(\mathcal{S})$. We next note that if $\nu(\{ \mathsf{np} \}) > 0$ we have $E_{\mathcal{V}} (\nu) = \infty$ (due to the $\log \| \vec{x} - \vec{y} \|_2^{-1}$ term in (\ref{S2Ker})). The latter means that $ \left\{ \nu \in \mathcal{M}(\Delta_{\mathcal{S}}) : E_{\mathcal{V}}(\nu) \leq \alpha  \right\}$ is a compact subset of $\{ \nu \in \mathcal{M}(\Delta_\mathcal{S}): \nu(\{ \mathsf{np}\}) = 0\}$ for any $\alpha \in \mathbb{R}$.

From (\ref{S2EFEq}) we know that $E_V(\mu) = E_{\mathcal{V}}(T_*\mu)$ for all $\mu \in \mathcal{M}(\mathbb{R})$. Combining the latter with Lemma \ref{S2TDiff} we conclude for $\alpha \in \mathbb{R}$ that 
\begin{equation}\label{S2FH1}
T_* \left\{ \mu \in \mathcal{M} (\Delta) : E_V(\mu) \leq \alpha \right\} = \left\{ \nu \in \mathcal{M}(\Delta_{\mathcal{S}}) : E_{\mathcal{V}}(\nu) \leq \alpha \mbox{ and }\nu(\{ \mathsf{np}\}) = 0 \right\},
\end{equation}
and so $  \left\{ \nu \in \mathcal{M}(\Delta_{\mathcal{S}}) : E_{\mathcal{V}}(\nu) \leq \alpha \mbox{ and }\nu(\{ \mathsf{np}\}) = 0 \right\}$ is homeomorphic to $ \left\{ \mu \in \mathcal{M}(\Delta) : E_{V}(\mu) \leq \alpha  \right\}$. As the right side of (\ref{S2FH1}) is a compact subset of $\{ \nu \in \mathcal{M}(\Delta_\mathcal{S}): \nu(\{ \mathsf{np}\}) = 0\}$ (we explained this in the previous paragraph), we conclude that $\left\{ \mu \in \mathcal{M} (\Delta) : E_V(\mu) \leq \alpha \right\}$ is a compact subset of $\mathcal{M}(\Delta)$. This proves that $E_V$ has compact level sets in $\mathcal{M}(\Delta)$.

Since $E_V(\mu)=E_{\mathcal{V}}(T_*\mu)$ by (\ref{S2EFEq}), we have by (\ref{S2Convex}) for $t \in (0,1)$ and $\mu, \nu \in \mathcal{M}(\mathbb{R})$ that
$$E_{V}(t\mu+(1-t)\nu) \hspace{-0.2mm} = \hspace{-0.2mm} E_{\mathcal{V}}(tT_*\mu+(1-t)T_* \nu)\hspace{-0.2mm} \leq \hspace{-0.2mm} tE_{\mathcal{V}}(T_*\mu)+(1-t)E_{\mathcal{V}}(T_*\nu) \hspace{-0.2mm} = \hspace{-0.2mm} tE_{V}(\mu) + (1- t) E_{V}(\nu), $$
with strict inequality if $\mu \neq \nu$. We conclude that $E_V(\mu)$ is strictly convex on $\mathcal{M}(\mathbb{R})$. 

Note that $\mathcal{M}_{\theta}(\Delta)$ is a closed, convex subset of $\mathcal{M}(\Delta)$ and from above $\left\{ \mu \in \mathcal{M} (\Delta) : E_V(\mu) \leq \alpha \right\}$ is compact for each $\alpha \in \mathbb{R}$, and also convex by the convexity of $E_V(\mu)$. The latter observation shows that for each $\alpha \in \mathbb{R}$ the set $\left\{\mu \in \mathcal{M}_{\theta} (\Delta) : E_V(\mu) \leq \alpha \right\}$
is a compact, convex subset of $\mathcal{M}_{\theta}(\Delta)$.

Summarizing the above, we have that $E_V(\mu)$ is strictly convex on the non-empty convex set $\mathcal{M}_{\theta}(\Delta)$, has compact and convex level sets in $\mathcal{M}_{\theta} (\Delta)$ and is not identically equal to $\infty$. This implies the existence and uniqueness of its minimizer $\me^{\theta} \in \mathcal{M}_{\theta}(\Delta)$. This proves part (2).\\

{\bf \raggedleft Step 2.} In this step we prove parts (3) and (4). 

From our work in Step 1, we know that $E_V(\mu)$ has compact level sets in $\mathcal{M}_\theta(\Delta)$ and since $\mathcal{M}_{\theta}(\Delta)$ is a non-empty closed subset of $\mathcal{M}(\mathbb{R})$ we conclude that $I_V^{\theta}(\mu)$ has compact level sets. In addition, since $F_V^{\theta}$ is finite from Step 1, we conclude that $I_V^{\theta}(\mu) \in [0, \infty]$. This proves that $I_V^{\theta}$ is a good rate function and completes part (3). In the remainder of this step we prove part (4).\\

Let $\mu \in \mathcal{M}_{\theta}(\Delta)$ be such that $E_V(\mu) < \infty$ and there is a constant $c \in \mathbb{R}$ such that 
\begin{equation}\label{S2DSIneq}
\begin{split}
&\int_\R \left(\log|x-y|^{-1}+\frac{1}{2}\log( 1+x^2) \right)\mu(dx) + \frac{1}{2} V(y) \geq c \mbox{, for a.e. $y \in \mathsf{Supp}( \theta^{-1} \lambda - \mu)$} \cap \Delta, \\
&\int_\R \left(\log|x-y|^{-1}+\frac{1}{2}\log( 1+x^2) \right)\mu(dx) + \frac{1}{2} V(y) \leq c \mbox{, for a.e. $y \in \mathsf{Supp}(\mu)$},
\end{split}
\end{equation}
where a.e. is with respect to the Lebesgue measure on $\mathbb{R}$. We seek to prove that $\mu = \me^{\theta}$. We mention that the integral appearing in (\ref{S2DSIneq}) is well-defined and takes value in $(-\infty, \infty]$, since the integrand is lower bounded in $x$ on $\mathbb{R}$ for each $y \in \mathbb{R}$ and $\mu$ is a probability measure. The existence of $c$, the continuity of $V$ and the fact that $\mu$ is a probability measure with density bounded by $\theta^{-1}$ together imply that the integral in (\ref{S2DSIneq}) is finite for each $y \in \mathbb{R}$.

By a direct computation using the definition of $T$ in Section \ref{Section2.2} and (\ref{S2Push}) we have for all $y \in \mathbb{R}$
\begin{equation*}
\begin{split}
&\int_\R \left(\log|x-y|^{-1}+\frac{1}{2}\log( 1+x^2) \right)\mu(dx) + \frac{1}{2} V(y) = \int_{\mathcal{S}} \log \| \vec{x} - T(y) \|^{-1}_2 (T_*\mu)(d\vec{x}) + \frac{1}{2} \mathcal{V}(T(y)),
\end{split}
\end{equation*}
where $\mathcal{V}(\vec{y})$ is as in (\ref{S2VtoV}). The last equation and the linearity of $T_*$ shows (\ref{S2DSIneq}) is equivalent to 
\begin{equation}\label{S2DSIneq2}
\begin{split}
& \int_{\mathcal{S}} \log \| \vec{x} - \vec{y} \|^{-1}_2 (T_*\mu)(d\vec{x}) + \frac{1}{2} \mathcal{V}(\vec{y}) \geq c \mbox{, for a.e. $\vec{y} \in \mathsf{Supp}(  \theta^{-1} T_* \lambda-T_*\mu)$} \cap T(\Delta), \\
& \int_{\mathcal{S}} \log \| \vec{x} - \vec{y} \|^{-1}_2 (T_*\mu)(d\vec{x})  +\frac{1}{2} \mathcal{V}(\vec{y})\leq c \mbox{, for a.e. $\vec{y} \in \mathsf{Supp}(T_*\mu) \cap T(\Delta)$},
\end{split}
\end{equation}
where the a.e. refers to the uniform measure on $\mathcal{S}$.

Let us write $\nu = T_*\mu$ and recall that $E_V(\mu)=E_{\mathcal{V}}(\nu)$ by (\ref{S2EFEq}). In particular, we have $  E_{\mathcal{V}}(\nu) < \infty$, which together with the fact that $\mathcal{V}$ is lower bounded implies that 
\begin{equation}\label{S2Split1}
I(\nu, \nu) < \infty, \hspace{2mm}\int_{\mathcal{S}}\mathcal{V}(\vec{y}) \nu(d\vec{y})< \infty \mbox{ and } E_{\mathcal{V}}(\nu) = I(\nu,\nu) + \int_{\mathcal{S}}\mathcal{V}(\vec{y}) \nu(d\vec{y}),
\end{equation}
where we recall $I(\mu,\nu) $ was defined in (\ref{S2MixedFunct}). Setting $\nu_{\mathsf{eq}} = T_* \me^{\theta}$ and using that $E_{\mathcal{V}}(\nu_{\mathsf{eq}})  = E_V(\me^{\theta}) = F_V^{\theta} < \infty$, we see that (\ref{S2Split1}) holds with $\nu = \nu_{\mathsf{eq}}$ as well. Using (\ref{S2Split1}) and (\ref{S2Ineq}), we have
\begin{equation*}
\begin{split}
&E_{\mathcal{V}}(\nu_{\mathsf{eq}}) - E_{\mathcal{V}}(\nu) = I(\nu_{\mathsf{eq}}, \nu_{\mathsf{eq}}) + I(\nu, \nu) + 2 \int_{\mathcal{S}} \left[ \int_{\mathcal{S}}  \log \| \vec{x} - \vec{y} \|^{-1}_2 \nu(d\vec{x}) +   \frac{1}{2} \mathcal{V}(\vec{y}) \right]  (\nu_{\mathsf{eq}} - \nu)(d\vec{y})   \\
&- 2 I( \nu_{\mathsf{eq}}, \nu)  \geq 2 \int_{\mathcal{S}} \left[ \int_{\mathcal{S}}  \log \| \vec{x} - \vec{y} \|^{-1}_2 \nu(d\vec{x}) +   \frac{1}{2} \mathcal{V}(\vec{y}) \right]  (\nu_{\mathsf{eq}} - \nu)(d\vec{y}).
\end{split}
\end{equation*}

To complete the proof it suffices to show that 
\begin{equation}\label{S2FC1}
\begin{split}
 \int_{\mathcal{S}} \left[ \int_{\mathcal{S}}  \log \| \vec{x} - \vec{y} \|^{-1}_2 \nu(d\vec{x}) +   \frac{1}{2} \mathcal{V}(\vec{y}) \right]  (\nu_{\mathsf{eq}} - \nu)(d\vec{y}) \geq 0.
\end{split}
\end{equation}
Indeed, if (\ref{S2FC1}) holds then the last two equations show that $E_{\mathcal{V}}(\nu_{\mathsf{eq}}) - E_{\mathcal{V}}(\nu) \geq 0$, which by (\ref{S2EFEq}) implies $E_V(\me^{\theta}) \geq E_V(\mu)$. As $\me^{\theta}$ is the unique minimizer of $E_V$ over $\mathcal{M}_{\theta}(\Delta)$ and $\mu \in \mathcal{M}_{\theta}(\Delta)$ by assumption we conclude that $\mu = \me^{\theta}$. 

Let us denote 
$$f(\vec{y}) = \int_{\mathcal{S}}  \log \| \vec{x} - \vec{y} \|^{-1}_2 \nu(d\vec{x}) +   \frac{1}{2} \mathcal{V}(\vec{y}) - c.$$
Using that $\int_{\mathcal{S}} c (\nu_{\mathsf{eq}} - \nu)(d\vec{y})  = 0 $, that $\mathsf{Supp}(\nu), \mathsf{Supp}(\nu_{\mathsf{eq}})  \subseteq \Delta_{\mathcal{S}}$, and the fact that $\nu(\{ \mathsf{np}\}) = 0 = \nu_{\mathsf{eq}}(\{ \mathsf{np}\})$ (see Lemma \ref{S2TDiff}), we see that 
\begin{equation}\label{S2FC2}
\begin{split}
& \int_{\mathcal{S}} \left[ \int_{\mathcal{S}}  \log \| \vec{x} - \vec{y} \|^{-1}_2 \nu(d\vec{x}) +   \frac{1}{2} \mathcal{V}(\vec{y}) \right]  (\nu_{\mathsf{eq}} - \nu)(d\vec{y}) = I_1 + I_2 \mbox{, where }\\
&I_1 = \int_{E_1} f(\vec{y})  (\nu_{\mathsf{eq}} - \nu)(d\vec{y}) , \hspace{2mm} I_2 = \int_{E_2} f(\vec{y})  (\nu_{\mathsf{eq}} - \nu)(d\vec{y}) \mbox{, with }\\
&E_1 = \{ \vec{y} \in T(\Delta): f(\vec{y}) > 0 \} \mbox{ and }E_2 = \{ \vec{y} \in T(\Delta): f(\vec{y}) < 0 \}.
\end{split}
\end{equation}
Since $\me^{\theta} \in \mathcal{M}_{\theta}(\Delta)$ we have $\theta^{-1} T_* \lambda \geq \nu_{\mathsf{eq}}$. In addition, from (\ref{S2DSIneq2}) we have $(  \theta^{-1} T_* \lambda - \nu)(E_2) = 0$. Combining the last two statements we get
\begin{equation}\label{S2FC3}
\begin{split}
I_2 = \int_{E_2} f(\vec{y})  (\nu_{\mathsf{eq}} - \theta^{-1} T_* \lambda + \theta^{-1} T_* \lambda - \nu)(d\vec{y}) = \int_{E_2} f(\vec{y})  (\nu_{\mathsf{eq}} - \theta^{-1} T_* \lambda )(d\vec{y}) \geq 0.
\end{split}
\end{equation}
Also, from (\ref{S2DSIneq2}) we have $\nu(E_1) = 0$, and so
\begin{equation}\label{S2FC4}
\begin{split}
I_1= \int_{E_1} f(\vec{y})  (\nu_{\mathsf{eq}} - \nu)(d\vec{y}) =\int_{E_1} f(\vec{y})  \nu_{\mathsf{eq}}(d\vec{y}) \geq 0.
\end{split}
\end{equation}
Combining (\ref{S2FC2}), (\ref{S2FC3}) and (\ref{S2FC4}) we obtain (\ref{S2FC1}). This suffices for the proof.

%-------------------------------------------------------------------------------------------------------------------------------------------------------------------------------------------------
% Section 3
%
%-------------------------------------------------------------------------------------------------------------------------------------------------------------------------------------------------
\section{LDP for $\mu_N$}\label{Section3} The goal of this section is to prove Theorem \ref{ThmMain}. The proof is presented in Section \ref{Section3.1} and relies on a certain weak LDP upper bound for the pushforward measures of $\mu_N$ under the map $T$ in Section \ref{Section2.2}, this is Lemma \ref{S3WLDP}, as well as two technical results -- Lemmas \ref{ConvL} and \ref{CompactL}. Lemma \ref{S3WLDP} is proved in Section \ref{Section3.2} by adapting some of the arguments from \cite{fe, Hardy, jo}, while Lemmas \ref{ConvL} and \ref{CompactL} are proved in Section \ref{Section3.3}. We continue with the same notation as in Sections \ref{Section1.2} and \ref{Section2}.

%-------------------------------------------------------------------------------------------------------------------------------------------------------------------------------------------------
% Section 3.1
%
%-------------------------------------------------------------------------------------------------------------------------------------------------------------------------------------------------
\subsection{Proof of Theorem \ref{ThmMain}}\label{Section3.1} We begin this section by stating some results, which will be used in the proof of Theorem \ref{ThmMain}. 

The following lemma establishes a weak LDP upper bound for the measures $\{T_* \mu_N\}_{ N \geq 1}$, where $T_*$ is as in (\ref{S2Push}) and $\mu_N$ are as in (\ref{S1DEM}), and is proved in Section \ref{Section3.2}.
\begin{lemma}\label{S3WLDP} Continue with the same notation from Theorem \ref{ThmMain} and suppose that the same assumptions hold. Define
\begin{equation}\label{S3PF}
Z_N' = Z_N \cdot N^{-N(N-1)\theta},
\end{equation}
where $Z_N$ is as in (\ref{DBEPF}). Then, for any $\mu \in \mathcal{M}(\mathcal{S})$ we have
\begin{equation}\label{WLDPUB}
\limsup_{\delta \rightarrow 0+} \limsup_{N \rightarrow \infty} \frac{1}{N^2} \log \left( Z_N' \mathbb{P}_N^{\theta} \left( T_* \mu_N \in B(\mu,\delta) \right) \right) \leq \begin{cases}  - \theta \cdot E_{\mathcal{V}}(\mu) &\mbox{ if $\mu \in T_*( \mathcal{M}_{\theta}(\Delta))$ } \\ -\infty &\mbox{ if  $\mu \not \in T_*( \mathcal{M}_{\theta}(\Delta))$ }\end{cases},
\end{equation}
where $B(\mu,\delta) = \{\rho \in \mathcal{M}(\mathcal{S}): d_2(\rho, \mu) < \delta\}$ (here $d_2$ is L{\'e}vy metric as in (\ref{LevyMetric})), $\mu_N$ are as in (\ref{S1DEM}), $T_*$ is as in (\ref{S2Push}) and $E_{\mathcal{V}}$ is as in (\ref{S2EF}) with $\mathcal{V}$ as in (\ref{S2VtoV}). In (\ref{WLDPUB}) we use $\log 0 = -\infty$.
\end{lemma}

We next state two technical lemmas, which will be required. They are proved in Section \ref{Section3.3}.
\begin{lemma}\label{ConvL}
Let $\theta > 0$, $A > \theta/2$, $\mu^{\infty} \in \mathcal{M}_{\theta}([-A,A])$. Suppose that $\ell^N \in \mathbb{W}_N^{\theta}(-\infty, \infty)$ are such that if $\mu^N = N^{-1} \sum_{i = 1}^N \delta_{\ell^N_i/N}$ we have
\begin{enumerate}
\item $\lim_{N \rightarrow \infty} d_1(\mu^N, \mu^{\infty}) = 0$, where $d_1$ is the L{\'e}vy metric on $\mathcal{M}(\mathbb{R})$ as in (\ref{LevyMetric});
\item $\mu_N \in \mathcal{M}([-A,A])$.
\end{enumerate}  
Let $V_N$ be continuous functions on $[-A,A]$ for $N \in \mathbb{N} \cup \{\infty\}$, such that $\lim_{N \rightarrow \infty} \sup_{[-A,A]}|V_N(x) - V_{\infty}(x)| = 0$.
Then, we have
\begin{equation}\label{JK1}
\lim_{N \rightarrow \infty} \iint_{\mathbb{R}^2} {\bf 1}\{ x \neq y \} k_{V_N}(x,y) \mu^N(dx) \mu^N(dy) = E_V(\mu^{\infty}),
\end{equation}
where $k_V$ and $E_V$ are as in (\ref{S1EF}).
\end{lemma}

\begin{lemma}\label{CompactL}
Let $a \in [-\infty, \infty)$ and $b \in (-\infty, \infty]$ be such that $a < b$, $\theta > 0$, and set $\Delta = [a,b + \theta] \cap \mathbb{R}$. Let $V$ be a continuous function on $\mathbb{R}$ that satisfies (\ref{S1VGrowth}). If $\mu \in \mathcal{M}_{\theta}(\Delta)$, then we can find a sequence of measures $\mu_n \in \mathcal{M}_{\theta}(\mathbb{R})$ such that
\begin{enumerate}
\item $\lim_{n \rightarrow \infty} d_1(\mu_n, \mu) = 0$, where $d_1$ is the L{\'e}vy metic on $\mathcal{M}(\mathbb{R})$ as in (\ref{LevyMetric});
\item $\lim_{n \rightarrow \infty} E_V(\mu_n) = E_V(\mu)$, where $E_V$ is as in (\ref{S1EF});
\item for each $n \in \mathbb{N}$, we have that $\mathsf{Supp}(\mu_n)$ is compact and contained in the interior of $\Delta$. 
\end{enumerate} 
\end{lemma}

With the above results in place we are ready to prove Theorem \ref{ThmMain}.
\begin{proof}[Proof of Theorem \ref{ThmMain}] The proof we present here is an adaptation of the proof of \cite[Theorem 1.1(c,d)]{Hardy}.
For clarity, we split the proof into three steps. \\

{\bf \raggedleft Step 1.} Note that it is enough to show that for any closed set $\mathcal{F} \subseteq \mathcal{M}(\mathbb{R})$,
\begin{equation}\label{GH1}
\limsup_{N \rightarrow \infty} \frac{1}{N^2} \log \left( Z_N' \mathbb{P}_N^{\theta} \left( \mu_N \in \mathcal{F} \right) \right) \leq  - \theta \cdot  \inf_{\mu \in \mathcal{F} \cap \mathcal{M}_{\theta}(\Delta) } E_V(\mu),
\end{equation}
and for any open set $\mathcal{O} \subseteq \mathcal{M}(\mathbb{R})$,
\begin{equation}\label{GH2}
\liminf_{N \rightarrow \infty} \frac{1}{N^2} \log \left( Z_N' \mathbb{P}_N^{\theta} \left( \mu_N \in \mathcal{O} \right) \right) \geq -  \theta \cdot \inf_{\mu \in \mathcal{O} \cap \mathcal{M}_{\theta}(\Delta) } E_V(\mu),
\end{equation}
where $Z_N'$ is as in (\ref{S3PF}). Indeed, if we take $\mathcal{F} = \mathcal{O} = \mathcal{M}(\mathbb{R})$ in (\ref{GH1}) and (\ref{GH2}) we get
$$\lim_{N \rightarrow \infty} \frac{1}{N^2} \log Z_N' = - \theta \cdot \inf_{\mu \in \mathcal{M}_{\theta}(\Delta)} E_V(\mu) = - \theta F_V^{\theta},$$
where the latter was defined in Theorem \ref{ThmFunct} and is finite. Combining the last equality with (\ref{GH1}), (\ref{GH2}), and the definition of $I_V^{\theta}$ in (\ref{S1Rate}), we conclude the statement of the theorem.

In the remainder of this step we establish (\ref{GH1}) and in Step 2 we prove (\ref{GH2}). The approach we take to proving (\ref{GH1}) is inspired by the proof of \cite[Theorem 4.1.1]{DZ}.
\\

Since (\ref{GH1}) is clear when $\mathcal{F} = \emptyset$, we assume that $\mathcal{F} \neq \emptyset$ is a closed subset of $\mathcal{M}(\mathbb{R})$. Then,
\begin{equation}\label{GH3}
\mathbb{P}_N^{\theta} \left( \mu_N \in \mathcal{F} \right) \leq \mathbb{P}_N^{\theta} \left( T_* \mu_N \in \mathrm{clo} ( T_*\mathcal{F}) \right), 
\end{equation}
where $\mathrm{clo}(T_*\mathcal{F})$ is the closure of $T_*\mathcal{F}$ in $\mathcal{M}(\mathcal{S})$. Let us fix $\epsilon > 0$ and introduce 
$$E_{\mathcal{V}}^\epsilon(\mu) = \begin{cases} \min \left( E_{\mathcal{V}}(\mu) - \epsilon, \epsilon^{-1} \right) &\mbox{ if } \mu \in T_*( \mathcal{M}_{\theta}(\Delta)), \\ \epsilon^{-1}& \mbox{ if } \mu \not \in T_*( \mathcal{M}_{\theta}(\Delta)).  \end{cases}$$
Then, from Lemma \ref{S3WLDP} for every $\mu \in \mathcal{M}(\mathcal{S})$ we can find $\delta_\mu > 0$ such that
\begin{equation}\label{GH4}
\limsup_{N \rightarrow \infty} \frac{1}{N^2} \log \left( Z_N' \mathbb{P}_N^{\theta} \left( T_* \mu_N \in B(\mu,\delta_{\mu}) \right) \right) \leq  - \theta E_{\mathcal{V}}^\epsilon(\mu).
\end{equation}
Since $\mathcal{M}(\mathcal{S})$ is compact, so is $\mathrm{clo}(T_*\mathcal{F})$, and thus we can find a finite number of measures $\nu_1, \dots, \nu_d \in \mathrm{clo}(T_*\mathcal{F})$, such that 
$$\mathbb{P}_N^{\theta} \left( T_* \mu_N \in \mathrm{clo} ( T_*\mathcal{F}) \right) \leq \sum_{i = 1}^d \mathbb{P}_N^{\theta} \left( T_* \mu_N \in B(\nu_i,\delta_{\nu_i}) \right).$$
Combining the latter with (\ref{GH3}) and (\ref{GH4}), we conclude that 
\begin{equation}\label{GH5}
\begin{split}
&\limsup_{N \rightarrow \infty} \frac{1}{N^2} \log \left( Z_N' \mathbb{P}_N^{\theta} \left( \mu_N \in \mathcal{F} \right) \right) \leq \max_{i = 1}^d \limsup_{N \rightarrow \infty} \frac{1}{N^2} \log \left( Z_N' \mathbb{P}_N^{\theta} \left( T_* \mu_N \in B(\nu_i,\delta_{\nu_i}) \right) \right) \\
&\leq  -\theta  \min_{i =1,\dots, d}  E_{\mathcal{V}}^\epsilon(\nu_i) \leq - \theta \inf_{\nu \in \mathrm{clo} ( T_*\mathcal{F}) } E_{\mathcal{V}}^{\epsilon}(\nu).
\end{split}
\end{equation}
Letting $\epsilon \rightarrow 0+$ in (\ref{GH5}), we obtain
\begin{equation*}
\begin{split}
&\limsup_{N \rightarrow \infty} \frac{1}{N^2} \log \left( Z_N' \mathbb{P}_N^{\theta} \left( \mu_N \in \mathcal{F} \right) \right) \leq - \theta \inf_{\nu \in \mathrm{clo} ( T_*\mathcal{F}) \cap T_*(\mathcal{M}_{\theta}(\Delta)) } E_{\mathcal{V}}(\nu) .
\end{split}
\end{equation*}
Finally, since $T_*$ is a homeomorphism between $\mathcal{M}(\mathbb{R})$ and $\{ \nu \in \mathcal{M}(\mathcal{S}): \nu (\{\mathsf{np}\}) = 0 \}$, we have
$$   \inf_{\nu \in \mathrm{clo} ( T_*\mathcal{F}) \cap T_*(\mathcal{M}_{\theta}(\Delta)) } E_{\mathcal{V}}(\nu) = \inf_{\nu \in  T_*\mathcal{F} \cap T_*(\mathcal{M}_{\theta}(\Delta)) } E_{\mathcal{V}}(\nu) =  \inf_{\mu \in  \mathcal{F} \cap \mathcal{M}_{\theta}(\Delta)} E_{V}(\mu),$$
where the last equality used (\ref{S2EFEq}). The last two equations imply (\ref{GH1}).\\

{\bf \raggedleft Step 2.} In this step we prove (\ref{GH2}). Note that it suffices to show that for each $\mu \in \mathcal{M}_{\theta}(\Delta)$ and open neighborhood $\mathcal{G} \subseteq \mathcal{M}(\mathbb{R})$, containing $\mu$, we have
\begin{equation}\label{GH6}
\liminf_{N \rightarrow \infty} \frac{1}{N^2} \log \left( Z_N' \mathbb{P}_N^{\theta} \left( \mu_N \in \mathcal{G} \right) \right) \geq -  \theta E_V(\mu).
\end{equation}
Note that by Lemma \ref{Sandwich} we have for all $N \in \mathbb{N}$ and $\ell \in \mathbb{W}_N^{\theta}(a_N,b_N)$
\begin{equation}\label{GH7}
\begin{split}
&Z_N' \mathbb{P}_N^{\theta} \left( \ell \right) = \exp\left( O \left(\sum_{ 1 \leq i < j \leq N} \frac{1}{(j-i)}  \right)  \right) \prod_{1 \leq i < j \leq N} \left(\ell_i/N - \ell_j/N \right)^{2\theta} \prod_{i = 1}^N e^{-\theta N V_N(\ell_i/N)} \\
& =  \exp\left( O \left(N \log N  \right)  \right) \prod_{1 \leq i < j \leq N} \left(\ell_i/N - \ell_j/N \right)^{2\theta} \prod_{i = 1}^N e^{-\theta N V_N(\ell_i/N)} \\
& =  \exp\left( O \left(N \log N  \right) - \theta N^2 \iint_{\mathbb{R}^2}{\bf 1}\{x \neq y\} k_{V_N}(x,y) \mu_N(dx) \mu_N(dy) - \theta \sum_{i = 1}^N V_N(\ell_i/N)  \right),
\end{split}
\end{equation}
where the constants in the big $O$ notations depend on $\theta$ alone, and may be different for different lines. We remark that in the first equality we used that $\ell_i - \ell_j \geq (j-i) \theta$ for $N \geq j > i \geq 1$.

If $\Delta = [a, a + \theta]$, i.e. $a = b$, then we have that $\mathcal{M}_{\theta}([a, a+\theta])$ contains a single element -- the uniform measure on $[a, a+ \theta]$. Thus $\mu$ has density $\theta^{-1} \cdot {\bf 1}\{x \in [a,a+\theta]\}$. Let $\ell^N \in \mathbb{W}_N^{\theta}(a_N, b_N)$ be given by $\ell^N_i = a_N + (N-i)\theta$ for $i = 1, \dots, N$, and set $\mu^N = N^{-1} \sum_{i = 1}^N \delta_{\ell^N_i/N}$. Since $\lim_{N \rightarrow \infty} N^{-1} a_N = a$, we conclude that $\mu^N$ weakly converge to $\mu$ (say by the Portmanteau theorem). In addition, we see that we can find a sufficiently large $A > 0$ so that $\mu^N$ satisfy the conditions of Lemma \ref{ConvL} ($\mu^{\infty} = \mu$ and $V_{\infty} = V$ here). From Lemma \ref{ConvL} and (\ref{GH7}) we conclude that
\begin{equation*}
\begin{split}
&\liminf_{N \rightarrow \infty} \frac{1}{N^2} \log \left( Z_N' \mathbb{P}_N^{\theta} \left( \mu_N \in \mathcal{G} \right) \right) \geq \liminf_{N \rightarrow \infty} \frac{1}{N^2} \log \left( Z_N' \mathbb{P}_N^{\theta} \left( \ell^N \right) \right) =  -\theta E_V(\mu),
\end{split}
\end{equation*}
which proves (\ref{GH6}) when $a = b$.\\

In the sequel we assume that $\Delta = [a, b + \theta]$ with $a < b$. In view of Lemma \ref{CompactL}, we see that it suffices to prove (\ref{GH6}) under the additional assumption that $\mathsf{Supp}(\mu)$ is compact and contained in the interior of $\Delta$. In particular, we assume that there exist $c,d \in \mathbb{R}$, with $c \leq d$ and $\epsilon > 0$ such that 
$$a \leq c - \epsilon, \hspace{2mm} d +  \epsilon \leq b, \hspace{2mm}  \mathsf{Supp}(\mu) \subseteq [c,d+\theta].$$

{\bf \raggedleft Claim:} There exists a sequence $\ell^N \in \mathbb{W}_N^{\theta}(-\infty, \infty)$ such that if $\mu^N = N^{-1} \sum_{i = 1}^N \delta_{\ell_i^N/N}$ we have
\begin{enumerate}
\item $\lim_{N \rightarrow \infty} d_1(\mu^N, \mu) = 0$;
\item $\ell^N \in \mathbb{W}_N^{\theta}(a_N, b_N)$ for all large enough $N$;
\item $\mathsf{Supp}(\mu^N) \subseteq [c-\epsilon, d+ \theta + \epsilon]$ for all large enough $N$. 
\end{enumerate}
We prove the claim in the next step. Here we assume its validity and conclude the proof of (\ref{GH6}).\\

We observe that, since $\ell^N \in \mathbb{W}_N^{\theta}(a_N, b_N)$ for all large enough $N$ and $\lim_{N \rightarrow \infty} d_1(\mu^N, \mu) = 0$,
$$\liminf_{N \rightarrow \infty} \frac{1}{N^2} \log \left( Z_N' \mathbb{P}_N^{\theta} \left( \mu_N \in \mathcal{G} \right) \right) \geq \liminf_{N \rightarrow \infty} \frac{1}{N^2} \log \left( Z_N' \mathbb{P}_N^{\theta} \left( \ell^N \right) \right).$$
On the other hand, since $\mathsf{Supp}(\mu^N) \subseteq [c-\epsilon, d+ \theta + \epsilon]$ for all large enough $N$, we conclude that we can find a sufficiently large $A > 0$ so that $\mu^N$ satisfy the conditions of Lemma \ref{ConvL} ($\mu^{\infty} = \mu$ and $V_{\infty} = V$ here). From Lemma \ref{ConvL} and (\ref{GH7}), we conclude that
\begin{equation*}
\begin{split}
&\lim_{N \rightarrow \infty} \frac{1}{N^2} \log \left( Z_N' \mathbb{P}_N^{\theta} \left( \ell^N \right) \right) =  -\theta E_V(\mu).
\end{split}
\end{equation*}
The last two equations prove (\ref{GH6}) when $a < b$.\\

{\bf \raggedleft Step 3.} In this step we construct $\ell^N$ as in the claim in Step 2. Let us denote the density of $\mu$ by $f(x)$, and note that since $\mu \in \mathcal{M}_{\theta}([c,d+\theta])$ we may assume that $0 \leq f(x) \leq \theta^{-1}$ for all $x \in \mathbb{R}$, $f(x) = 0$ for $x \not \in [c,d + \theta]$. We let $y_i$, for $i =1, \dots, N$, denote the quantiles of $\mu$, defined as the smallest real numbers such that 
$$\int_{-\infty}^{y_i} f(x)dx = \frac{i -1/2}{N}.$$
Since $f(x) = 0$ for $x \not \in [c,d + \theta]$, we know that $y_i \in [c, d + \theta]$ for all $i = 1,\dots, N$. 

We now let $\ell^N_i$ denote the largest element in $\mathbb{Z} + (N-i) \theta$, which is less than or equal to $N y_{N-i+1}$. We claim that $\ell^N = (\ell^N_1, \dots, \ell^N_N) \in \mathbb{W}^N_{\theta}(-\infty, \infty)$, or equivalently we want 
$$\lambda_1^N \geq \cdots \geq \lambda_N^N \mbox{, where } \lambda_i^N = \ell_i^N - (N-i)\theta.$$
Suppose, for the sake of contradiction, that $\lambda_i^N - \lambda_{i-1}^N \geq 1$ for some $i \in \{2, \dots, N\}$. Then,
$$\ell_{i-1}^N + 1 = \lambda^N_{i-1} + 1 + (N-i+1)\theta \leq \ell_i^N + \theta \leq Ny_{N-i+1} + \theta = Ny_{N - i +2} + N(y_{N-i+1} - y_{N-i+2}) + \theta.$$
On the other hand, as $f(x) \in [0, \theta^{-1}]$, we have 
$$\frac{1}{N} = \int_{y_{N-i+1}}^{y_{N-i+2}} f(x) dx \leq \theta^{-1}  (y_{N-i+2} - y_{N-i+1}) \implies N (y_{N-i+1} - y_{N-i+2})\leq - \theta.$$
Combining the last two inequalities we get $\ell^N_{i-1} + 1 \leq Ny_{N - i +2}$, which contradicts the maximality of $\ell^N_{i-1}$. As we got our desired contradiction, we conclude that $\ell^N  \in \mathbb{W}^N_{\theta}(-\infty, \infty)$.\\

In the remainder of this step, we prove that $\ell^N$ satisfy the three conditions of the claim. We readily observe that $\mu^N$ weakly converge to $\mu$ (say by the Portmanteau theorem), which establishes the first statement. In addition, by construction we have
$$ Nc - 1 \leq N y_1 -1 \leq \ell_N^N \mbox{ and } \ell_1^N \leq N y_N \leq N(d+ \theta),$$
which readily establishes third statement in the claim. In addition, since $\lim_{N \rightarrow \infty} N^{-1} a_N = a \leq c - \epsilon$ and $\lim_{N \rightarrow \infty} N^{-1} b_N = b \geq d + \epsilon$, we see that the last inequality implies $a_N \leq \ell_N^N$ and $ \ell_1^N \leq b_N + (N-1)\theta$ for all large $N$, which proves the second statement in the claim.
\end{proof}

%-------------------------------------------------------------------------------------------------------------------------------------------------------------------------------------------------
% Section 3.2
%
%-------------------------------------------------------------------------------------------------------------------------------------------------------------------------------------------------
\subsection{Weak LDP upper bound for $\{ T_* \mu_N\}_{N \geq 1}$ }\label{Section3.2}  In this section we prove Lemma \ref{S3WLDP}. Our proof is split into three parts -- these are Sections \ref{Section3.2.1}, \ref{Section3.2.2} and \ref{Section3.2.3}. In Sections \ref{Section3.2.1} and \ref{Section3.2.2} we prove that for each $\mu \in \mathcal{M}(\mathcal{S})$ we have
\begin{equation}\label{FG1}
\limsup_{\delta \rightarrow 0+} \limsup_{N \rightarrow \infty} \frac{1}{N^2} \log \left( Z_N' \mathbb{P}_N^{\theta} \left( T_* \mu_N \in B(\mu,\delta) \right) \right) \leq - \theta \cdot E_{\mathcal{V}}(\mu),
\end{equation}
under Assumption 2(a), and Assumption 2(b), respectively. In Section \ref{Section3.2.3} we prove (\ref{WLDPUB}).

%-------------------------------------------------------------------------------------------------------------------------------------------------------------------------------------------------
% Section 3.2.1
%
%-------------------------------------------------------------------------------------------------------------------------------------------------------------------------------------------------
\subsubsection{Proof under Assumption 2(a)}\label{Section3.2.1} In this section we prove (\ref{FG1}) when Assumption 2(a) holds. The proof we present here is adapted from \cite[Proposition 2.3]{Hardy} and for clarity is split into two steps.\\

{\bf \raggedleft Step 1.} In this step, we introduce some relevant notation and establish a few technical estimates, which will be used in the next step. 

 If $T$ is as in Section \ref{Section2.2}, one directly verifies that 
\begin{equation}\label{Q1}
\|T(x) - T(y)\|_2 = \frac{|x-y|}{\sqrt{1 + x^2} \cdot \sqrt{1 + y^2}} \mbox{ for } x,y \in \mathbb{R}.
\end{equation}
From (\ref{VNgrowth}) with $N = 1$, we know that there exists $\theta' > 1/2$ such that 
$$\liminf_{|x| \rightarrow \infty} \theta V(x) - \theta' \log (1 + x^2) > - \infty.$$
The latter and the continuity of $V$ implies that there is a constant $A_1 \geq 0$ such that 
$$ - \theta' \log (1 + x^2) + A_1 \geq - \theta V(x) \mbox{ for all $x \in \mathbb{R}$}.$$
The latter inequality implies that for some $A_2 > 0$, depending on $\theta$ and $V$, and $N \geq 2$ 
\begin{equation}\label{Q2}
\sum_{ \ell \in \mathbb{W}_N^{\theta}(a_N, b_N)} e^{-\theta V(\ell_i/N)} \leq \prod_{i =1}^N \left(\sum_{x \in \mathbb{Z}} \exp \left( -\theta V\left(\frac{x + (N-i)\theta}{N}\right) \right) \right) \leq \exp \left(A_2 N \log N\right).
\end{equation}

If $\mathcal{V}$ is as in (\ref{S2VtoV}) and $F_{\mathcal{V}}$ as in (\ref{S2Ker}), we know that $F_{\mathcal{V}}$ is lower bounded and lower semi-continuous on $\mathcal{S} \times \mathcal{S}$. This ensures the existence of continuous functions $\{F_{\mathcal{V}}^M \}_{M \geq 1}$ such that $F_{\mathcal{V}}^M$ increase pointwise to $F_{\mathcal{V}}$. By replacing $F_{\mathcal{V}}^M$ with $\min( M, F_{\mathcal{V}}^M)$ we may also assume that $F_{\mathcal{V}}^M \leq M$. 

With the above notation we can proceed with the main argument in the next step.\\

{\bf \raggedleft Step 2.} Combining (\ref{GH7}) with (\ref{Q1}), and setting $\vec{z}_i = T(\ell_i/N)$ for $i =1, \dots, N$, we conclude that
\begin{equation*}
\begin{split}
Z_N' \mathbb{P}_N^\theta(T_{*}\mu_N\in B(\mu,\delta)) \leq & \hspace{1mm} e^{C_{\theta} N \log N + \theta \epsilon_N N^2} \sum_{\substack{\ell \in \mathbb{W}_N^{\theta}(a_N, b_N) : \\ T_* \mu_N \in B(\mu,\delta) }} \prod_{1 \leq i < j \leq N} \left\| \vec{z}_i- \vec{z}_j \right\|_2^{2\theta} \\
&\times \prod_{i = 1}^N e^{-\theta (N-1) (V(\ell_i/N) - \log (1 + (\ell_i/N)^2)) - \theta V(\ell_i/N)} ,
\end{split}
\end{equation*}
where $C_{\theta}$ is a positive constant that depends on $\theta$ alone, and $\epsilon_N$ are as in Assumption 2(a). If $\{F_{\mathcal{V}}^M \}_{M \geq 1}$ are as in Step 1, then we see that the last inequality implies for $N \geq 2$ and $M \in \mathbb{N}$ 
\begin{equation}\label{S3E2}
\begin{split}
&Z_N' \mathbb{P}_N^\theta(T_{*}\mu^N\in B(\mu,\delta))  \\
& \leq e^{C_{\theta} N \log N + \theta \epsilon_N N^2} \sum_{\substack{\ell \in \mathbb{W}_N^{\theta}(a_N, b_N) : \\ T_* \mu_N \in B(\mu,\delta) }} \exp \left( - \theta \sum_{1 \leq i \neq j \leq N} F_{\mathcal{V}}(\vec{z}_i, \vec{z}_j)  \right) \prod_{i = 1}^N e^{-\theta V(\ell_i/N)}  \\
& \leq e^{C_{\theta} N \log N + \theta \epsilon_N N^2} \hspace{-6mm} \sum_{\substack{\ell \in \mathbb{W}_N^{\theta}(a_N, b_N) : \\ T_* \mu_N \in B(\mu,\delta) }} \hspace{-6mm} \exp \left( - \theta N^2 \iint_{\vec{x} \neq \vec{y}}  F^M_{\mathcal{V}}(\vec{x}, \vec{y}) T_* \mu_N(d\vec{x}) T_* \mu_N(d\vec{y})  \right) \prod_{i = 1}^N  e^{-\theta V(\ell_i/N)},
\end{split}
\end{equation}

Moreover, we have $\mathbb{P}_N^\theta$-almost surely 
$$T_* \mu_N \otimes T_* \mu_N ( \{ (\vec{x},\vec{y}) \in \mathcal{S} \times \mathcal{S}: \hspace{1mm} \vec{x} = \vec{y} \}) = \frac{1}{N},$$
which implies that on the event $\{ T_{*}\mu_N\in B(\mu,\delta)\}$ we have 
\begin{equation}\label{Q3}
\begin{split}
\iint_{\vec{x} \neq \vec{y}}  F^M_{\mathcal{V}}(\vec{x}, \vec{y}) T_* \mu_N(d\vec{x}) T_* \mu_N(d\vec{y}) &  \geq  \iint_{\mathcal{S}^2}  F^M_{\mathcal{V}}(\vec{x}, \vec{y}) T_* \mu_N(d\vec{x}) T_* \mu_N(d\vec{y}) - \frac{1}{N} \max_{\mathcal{S} \times \mathcal{S}}   F^M_{\mathcal{V}}  \\
& \geq \inf_{\nu \in B(\mu,  \delta)} \iint_{\mathcal{S}^2}  F^M_{\mathcal{V}}(\vec{x}, \vec{y}) \nu(d\vec{x}) \nu(d\vec{y}) - \frac{M}{N},
\end{split}
\end{equation}
where in the last inequality we used that $F^M_{\mathcal{V}} \leq M$. 

Combining the last inequality with (\ref{Q2}) and (\ref{S3E2}) we conclude that for some $B_M$, depending on $\theta$, $V$ and $M$, and $N \geq 2$
\begin{equation*}
\begin{split}
&N^{-2} \log \left( Z_N' \mathbb{P}_N^\theta(T_{*}\mu^N\in B(\mu,\delta)) \right) \\
&\leq  - \theta \inf_{\nu \in B(\mu,  \delta)} \iint_{\mathcal{S}^2}  F^M_{\mathcal{V}}(\vec{x}, \vec{y}) \nu(d\vec{x}) \nu(d\vec{y}) +   B_M \cdot  N^{-1} \log N + \theta \epsilon_N,
\end{split}
\end{equation*}
which implies that 
\begin{equation*}
\begin{split}
&\limsup_{N \rightarrow \infty} \frac{1}{N^2}  \log \left( Z_N' \mathbb{P}_N^\theta(T_{*}\mu^N\in B(\mu,\delta)) \right) \leq  - \theta \inf_{\nu \in B(\mu,  \delta)} \iint_{\mathcal{S}^2}  F^M_{\mathcal{V}}(\vec{x}, \vec{y}) \nu(d\vec{x}) \nu(d\vec{y}).
\end{split}
\end{equation*}
The last inequality and the continuity of $F^M_{\mathcal{V}}$ on $\mathcal{S} \times \mathcal{S}$ implies
\begin{equation*}
\begin{split}
&\limsup_{\delta \rightarrow 0+}\limsup_{N \rightarrow \infty} \frac{1}{N^2}  \log \left( Z_N' \mathbb{P}_N^\theta(T_{*}\mu^N\in B(\mu,\delta)) \right) \leq  - \theta  \iint_{\mathcal{S}^2}  F^M_{\mathcal{V}}(\vec{x}, \vec{y}) \mu(d\vec{x}) \mu(d\vec{y}).
\end{split}
\end{equation*}
Letting $M \rightarrow \infty$ in the last inequality and using that, by the monotone convergence theorem, the right side converges to $-\theta E_{\mathcal{V}}(\mu)$ we conclude (\ref{FG1}).

%-------------------------------------------------------------------------------------------------------------------------------------------------------------------------------------------------
% Section 3.2.2
%
%-------------------------------------------------------------------------------------------------------------------------------------------------------------------------------------------------
\subsubsection{Proof under Assumption 2(b)}\label{Section3.2.2} In this section we prove (\ref{FG1}) when Assumption 2(b) holds. In particular, we have that $V_N$ converge uniformly over compact subsets to $V$, and satisfy the growth condition (\ref{FeGrowth}). Notice that by the continuity of $V_N$ and the uniform convergence to $V$ over compacts, we may shift $V_N$ and $V$ by the same positive constant (which of course does not affect $\mathbb{P}_N^{\theta}$), so that
\begin{equation}\label{R1}
V_N(x) \geq (1 + \xi) \log (1 + x^2) \mbox{ for all $x \in \mathbb{R}$ and $N \geq 1$.} 
\end{equation}
The proof we present here is adapted from \cite{fe,jo} and for clarity is split into two steps.\\

{\bf \raggedleft Step 1.} In this step we introduce some relevant notation and establish a few technical estimates, which will be used in the next step. 

Using that $|x-y|^2 \leq (1 + x^2) (1 + y^2)$ and (\ref{R1}) we have that 
\begin{equation}\label{R2}
- k_{V_N}(x,y) =  \log |x - y| - \frac{1}{2} V_N(x) - \frac{1}{2} V_N(y) \leq - \frac{\xi}{2} \log (1 + x^2) - \frac{\xi}{2} \log (1 + y^2),
\end{equation}
where we recall that $k_V$ was introduced in (\ref{S1EF}). In addition, we have that there exists a constant $A_1 > 0$, depending on $\theta$ alone, such that for all $N \geq 2$ and $\alpha \geq 1$
\begin{equation}\label{R3}
\begin{split}
&\sum_{\ell \in \mathbb{W}_N^{\theta}(a_N, b_N)} \prod_{i = 1}^N \exp \left( - \alpha \log (1 +  (\ell_i/N)^2) \right) \\
&\leq \prod_{i = 1}^N \left( \sum_{x \in \mathbb{N}} \exp \left( - \alpha \log \left(1 + \frac{[x + (N-i)\theta ]^2}{N^2}\right) \right)\right)  \leq \exp (A_1 N \log N).
\end{split}
\end{equation}

Let $\mathcal{V}_N$ be as in (\ref{S2VtoV}) and $F_{\mathcal{V}_N}$ as in (\ref{S2Ker}) with $V$ replaced with $V_N$. For $M \in \mathbb{N}$ we define 
\begin{equation}\label{R4}
F^M_{\mathcal{V}_N}(\vec{x}, \vec{y}) = \min \left(\log \| \vec{x} - \vec{y} \|_2^{-1}, M/2 \right) + \frac{1}{2} \min \left( \mathcal{V}_N(\vec{x}), M/2\right) + \frac{1}{2} \min\left( \mathcal{V}_N(\vec{y}), M/2 \right),
\end{equation}
and note that from (\ref{R1}) and (\ref{S2VtoV}) there is an $M$ dependent neighborhood $U_M$ around the point $(\mathsf{np}, \mathsf{np}) \in \mathcal{S} \times \mathcal{S}$ such that $F^M_{\mathcal{V}_N}(\vec{x}, \vec{y}) =  M$ for $\vec{x}, \vec{y} \in U_M$. In particular, we conclude that $ F^M_{\mathcal{V}_N}$ are continuous on $\mathcal{S} \times \mathcal{S}$. In addition, from the uniform convergence of $V_N$ to $V$ over compact sets, we conclude that the sequence
\begin{equation}\label{R5}
a_N^M:= \sup_{ \vec{x}, \vec{y} \in \mathcal{S}} \left| F^M_{\mathcal{V}_N}(\vec{x}, \vec{y}) - F^M_{\mathcal{V}}(\vec{x}, \vec{y}) \right| \mbox{ satisfies } \lim_{N \rightarrow \infty} a_N^M = 0.
\end{equation}

With the above notation we can proceed with the main argument in the next step.\\

{\bf \raggedleft Step 2.} Let us fix $\rho \in (0,1)$. Combining (\ref{GH7}), (\ref{Q1}), (\ref{R1}) and (\ref{R2}), and setting $\vec{z}_i = T(\ell_i/N)$ for $i =1, \dots, N$, we conclude that for $M,N \geq 1$ we have
\begin{equation*}
\begin{split}
&Z_N' \mathbb{P}_N^\theta(T_{*}\mu_N\in B(\mu,\delta)) \leq  \hspace{1mm} e^{C_{\theta} N \log N} \sum_{\substack{\ell \in \mathbb{W}_N^{\theta}(a_N, b_N) : \\ T_* \mu_N \in B(\mu,\delta) }} \exp \left( - \theta (1- \rho) \sum_{1 \leq i \neq j \leq N} F_{\mathcal{V}_N}(\vec{z}_i, \vec{z}_j)  \right)  \\
&\times   \exp \left( - \theta \rho \sum_{1 \leq i \neq j \leq N} k_{V_N}(\ell_i/N, \ell_j/N) - \theta \sum_{i = 1}^N V_N(\ell_i/N)\right)  \leq e^{C_{\theta} N \log N} \sum_{\substack{\ell \in \mathbb{W}_N^{\theta}(a_N, b_N) : \\ T_* \mu_N \in B(\mu,\delta) }} \\
& \exp \left( - \theta  (1- \rho) N^2 \iint_{\vec{x} \neq \vec{y}}  F^M_{\mathcal{V}_N}(\vec{x}, \vec{y}) T_* \mu_N(d\vec{x}) T_* \mu_N(d\vec{y})  -  \theta\rho \xi (N-1)  \sum_{i = 1}^N \log (1 + (\ell_i/N)^2)  \right),
\end{split}
\end{equation*}
where $ F^M_{\mathcal{V}_N}(\vec{x}, \vec{y})$ are as in Step 2, and $C_{\theta}$ is a positive constant that depends on $\theta$ alone.

Arguing as in (\ref{Q3}) we have
\begin{equation*}
\begin{split}
\iint_{\vec{x} \neq \vec{y}}  F^M_{\mathcal{V}_N}(\vec{x}, \vec{y}) T_* \mu_N(d\vec{x}) T_* \mu_N(d\vec{y})  & \geq \inf_{\nu \in B(\mu,  \delta)} \iint_{\mathcal{S}^2}  F^M_{\mathcal{V}_N}(\vec{x}, \vec{y}) \nu(d\vec{x}) \nu(d\vec{y}) - \frac{M}{N} \\
& \geq  \inf_{\nu \in B(\mu,  \delta)} \iint_{\mathcal{S}^2}  F^M_{\mathcal{V}}(\vec{x}, \vec{y}) \nu(d\vec{x}) \nu(d\vec{y}) - a^M_N - \frac{M}{N} ,
\end{split}
\end{equation*}
where we recall that $F^M_{\mathcal{V}_N}$ were defined in Step 1, and satisfy $F^M_{\mathcal{V}_N} \leq M$, while $a_N^M$ are as in (\ref{R5}).

Combining the last two inequalities with (\ref{R3}) we conclude that for $N$ large enough so that $\theta\rho \xi (N-1)  \geq 1$, and $N \geq 2$, and $M \geq 1$ we have
\begin{equation*}
\begin{split}
Z_N' \mathbb{P}_N^\theta(T_{*}\mu_N\in B(\mu,\delta)) \leq &  \exp \left( - \theta  (1- \rho) N^2\inf_{\nu \in B(\mu,  \delta)} \iint_{\mathcal{S}^2}  F^M_{\mathcal{V}}(\vec{x}, \vec{y}) \nu(d\vec{x}) \nu(d\vec{y}) \right) \\
& \times \exp \left(  C_{\theta} N \log N + \theta (1-\rho) N^2 a^M_N + MN  + A_1 N \log N \right).
\end{split}
\end{equation*}
Using that $\lim_{N \rightarrow \infty} a_N^M = 0$, see (\ref{R5}), we conclude that
\begin{equation*}
\begin{split}
\limsup_{N \rightarrow \infty} \frac{1}{N^2}  \log \left( Z_N' \mathbb{P}_N^\theta(T_{*}\mu^N\in B(\mu,\delta)) \right) \leq - \theta  (1- \rho) \inf_{\nu \in B(\mu,  \delta)} \iint_{\mathcal{S}^2}  F^M_{\mathcal{V}}(\vec{x}, \vec{y}) \nu(d\vec{x}) \nu(d\vec{y}).
\end{split}
\end{equation*}
The last inequality and the continuity of $F^M_{\mathcal{V}}$ on $\mathcal{S} \times \mathcal{S}$ implies
\begin{equation*}
\begin{split}
&\limsup_{\delta \rightarrow 0+}\limsup_{N \rightarrow \infty} \frac{1}{N^2}  \log \left( Z_N' \mathbb{P}_N^\theta(T_{*}\mu^N\in B(\mu,\delta)) \right) \leq  - \theta (1- \rho) \iint_{\mathcal{S}^2}  F^M_{\mathcal{V}}(\vec{x}, \vec{y}) \mu(d\vec{x}) \mu(d\vec{y}).
\end{split}
\end{equation*}
We may now let $M \rightarrow \infty$ above, and note that the right side converges to $ - \theta (1-\rho) E_{\mathcal{V}}(\mu)$ by the monotone convergence theorem, and subsequently take $\rho \rightarrow 0+$ to get (\ref{FG1}).

%-------------------------------------------------------------------------------------------------------------------------------------------------------------------------------------------------
% Section 3.2.3
%
%-------------------------------------------------------------------------------------------------------------------------------------------------------------------------------------------------
\subsubsection{Proof of Lemma \ref{S3WLDP}}\label{Section3.2.3}  In this section we conclude the proof of Lemma \ref{S3WLDP}. For clarity we split the proof into two steps.\\

{\bf \raggedleft Step 1.} If $\mu \in T_*(\mathcal{M}_{\theta}(\Delta))$, then we have that (\ref{WLDPUB}) follows from (\ref{FG1}), which was established in Sections \ref{Section3.2.1} and \ref{Section3.2.2} above. We may thus assume that $\mu \in \mathcal{M}(\mathcal{S})$ and $\mu \not \in T_*(\mathcal{M}_{\theta}(\Delta))$.

If $\mu(\{\mathsf{np}\}) > 0$, then we have that $E_{\mathcal{V}}(\mu) = \infty$, because of the term $\log \| \vec{x} - \vec{y}\|_2^{-1}$ in $F_{\mathcal{V}}$, see (\ref{S2Ker}). In particular, we see that in this case (\ref{WLDPUB}) again follows from (\ref{FG1}), and we may assume that $\mu(\{\mathsf{np}\}) = 0$.

From Lemma \ref{S2TDiff} we know that $T_*$ is a homeomorphism between $\mathcal{M}(\mathbb{R})$ and $\{ \nu \in \mathcal{M}(\mathcal{S}): \nu (\{\mathsf{np}\}) = 0 \}$, and so there is a unique measure $\rho \in \mathcal{M}(\mathbb{R})$ such that $T_*\rho = \mu$. Since $\mu \not \in T_*(\mathcal{M}_{\theta}(\Delta))$ we know that $\rho \not \in \mathcal{M}_{\theta}(\Delta)$.

We claim that there exist $\epsilon_0 > 0$ and $N_0 \in \mathbb{N}$, such that for $N \geq N_0$ we have
\begin{equation}\label{PR1}
\mathbb{P}_N^{\theta} \left(  d_1(\rho, \mu_N) < \epsilon_0 \right)  = 0,
\end{equation}
where $d_n$ is the L{\'e}vy metric on $\mathcal{M}(\mathbb{R}^n)$ as in (\ref{LevyMetric}). We will prove (\ref{PR1}) in Step 2 below. For now, we assume its validity and conclude the proof of (\ref{WLDPUB}).\\

Since $T_*$ is a homeomorphism between $\mathcal{M}(\mathbb{R})$ and $\{ \nu \in \mathcal{M}(\mathcal{S}): \nu (\{\mathsf{np}\}) = 0 \}$, we conclude that there exists $\delta_0 > 0$ such that 
$$B(\mu, \delta_0) \cap \{ \nu \in \mathcal{M}(\mathcal{S}): \nu (\{\mathsf{np}\}) = 0 \} \subseteq T_* \{ \rho' \in \mathcal{M}(\mathbb{R}): d_1(\rho, \rho') < \epsilon_0\}.$$
The latter implies that for $\delta \in (0, \delta_0]$ and $N \geq N_0$
\begin{equation*}
\begin{split}
\mathbb{P}_N^{\theta} \hspace{-0.5mm} \left( T_* \mu_N \in B(\mu,\delta) \right) = \mathbb{P}_N^{\theta}\hspace{-0.5mm}  \left( T_* \mu_N \in B(\mu,\delta)\hspace{-0.5mm}  \cap \hspace{-0.5mm}  \{ \nu \in \mathcal{M}(\mathcal{S}) \hspace{-0.5mm} : \hspace{-0.5mm}  \nu (\{\mathsf{np}\}) = 0 \} \hspace{-0.5mm}   \right) \leq  \mathbb{P}_N^{\theta} \hspace{-0.5mm}  \left( d_1(\rho, \mu_N) < \epsilon_0   \right) = 0, 
\end{split}
\end{equation*}
which implies (\ref{WLDPUB}).\\

{\bf \raggedleft Step 2.} In this step we prove (\ref{PR1}). Observe that $\mathcal{M}_{\theta}(\Delta) = \mathcal{M}(\Delta) \cap \mathcal{M}_{\theta}(\mathbb{R})$ and both $ \mathcal{M}(\Delta)$ and $\mathcal{M}_{\theta}(\mathbb{R})$ are closed subsets of $\mathcal{M}(\mathbb{R})$. Since $\rho \not \in \mathcal{M}_{\theta}(\Delta) $ we conclude that there exists $\epsilon \in (0,1)$, such that at least one of the following holds:
\begin{enumerate}
\item $d_1(\rho, \rho') > 2\epsilon \mbox{ for all } \rho' \in \mathcal{M}_{\theta}(\mathbb{R})$,
\item  $d_1(\rho, \rho') > 2\epsilon \mbox{ for all } \rho' \in \mathcal{M}(\Delta) $.
\end{enumerate}

Suppose first that $d_1(\rho, \rho') > 2\epsilon \mbox{ for all } \rho' \in \mathcal{M}_{\theta}(\mathbb{R})$. We take $N_0 = \lceil \theta \epsilon^{-1} \rceil$ and proceed to prove (\ref{PR1}) with this choice of $N_0$ and $\epsilon_0 = \epsilon$. From (\ref{S1DEM}) we know that $\mu_N = \frac{1}{N} \sum_{i = 1}^N \delta_{\ell_i/N},$
and we let 
$$\tilde{\mu}_N(x) = \sum_{i = 1}^N \theta^{-1} \cdot {\bf 1} \{ x \in [\ell_i/N, \ell_i/N + \theta/N)\}.$$
Notice that $\tilde{\mu}_N$ is a probability density function on $\mathbb{R}$, and using the same letter to denote the corresponding measure we have $\tilde{\mu}_N \in \mathcal{M}_{\theta}(\mathbb{R})$. Here we implicitly used that $\ell_i - \ell_j \geq \theta (j-i)$ for $1 \leq i < j \leq N$. Furthermore, we have from (\ref{LevyMetric}) that 
$$d_1(\mu_N, \tilde{\mu}_N) \leq \theta N^{-1},$$
which implies that for $N \geq N_0$ we have 
$$\mathbb{P}_N^{\theta} \left(  d_1(\rho, \mu_N) < \epsilon \right) \leq  \mathbb{P}_N^{\theta} \left(  d_1(\rho, \tilde{\mu}_N) < \epsilon + \theta N^{-1} \right) \leq \mathbb{P}_N^{\theta} \left(  d_1(\rho, \tilde{\mu}_N) < 2\epsilon \right) = 0,$$
where the last equality used that $\tilde{\mu}_N \in \mathcal{M}_{\theta}(\mathbb{R})$ so that $ d_1(\rho, \tilde{\mu}_N) > 2\epsilon$. This establishes (\ref{PR1}) when $d_1(\rho, \rho') > 2\epsilon \mbox{ for all } \rho' \in \mathcal{M}_{\theta}(\mathbb{R})$.\\

Finally, we suppose that $d_1(\rho, \rho') > 2\epsilon \mbox{ for all } \rho' \in \mathcal{M}(\Delta) $. In particular, $\Delta \neq \mathbb{R}$ and so we have that either $\Delta = [a, \infty)$, $\Delta =(-\infty, b + \theta]$ or $\Delta = [a,b + \theta]$ for some finite $a, b$ with $a \leq b$. As the three cases are handled quite similarly, we will only consider the case when $\Delta = [a,b + \theta]$.

From Assumption 1, we know that $\lim_{N \rightarrow \infty} N^{-1} a_N = a$ and $\lim_{N \rightarrow \infty} N^{-1} b_N = b$. The latter implies that there is $N_0 \in \mathbb{N}$ such that for $N \geq N_0$ and $A_1 = \lfloor \epsilon N/4 \rfloor$, $B_1 = \lfloor (1-\epsilon/4) N \rfloor$ we have
\begin{equation}\label{PR2}
\begin{split}
1 \leq A_1 \leq B_1 \leq N, \hspace{1mm} \frac{b_N + (N-A_1)\theta}{N}  \leq b + \theta, \hspace{1mm} \frac{a_N + (N-B_1) \theta}{N} \geq a, \hspace{1mm} \frac{A_1 + N - B_1}{N} < \epsilon.
\end{split}
\end{equation}
Below we proceed to prove (\ref{PR1}) for this choice of $N_0$ and $\epsilon_0 = \epsilon$.

Throughout we assume that $N \geq N_0$. For $\ell \in \mathbb{W}^{\theta}_{N}(a_N, b_N)$ we define 
$$A = \begin{cases} \min \{i \in \{1, \dots, N\} : \ell_i  \leq Nb + N \theta \} &\mbox{ if } \ell_N \leq Nb + N \theta, \\ 0 &\mbox{ if } \ell_N > Nb + N \theta, \end{cases} \mbox{ and }$$
$$B = \begin{cases} \max \{i \in \{1, \dots, N\} : \ell_i  \geq Na  \} &\mbox{ if } \ell_1 \geq Na \\ 0 &\mbox{ if } \ell_1 < Na  \end{cases}.$$
Observe that from (\ref{PR2}) we have $1 \leq A \leq A_1 \leq B_1 \leq B \leq N$ for all $\ell \in \mathbb{W}^{\theta}_{N}(a_N, b_N)$. 

If $\mu_N = \frac{1}{N} \sum_{i = 1}^N \delta_{\ell_i/N}$ we define 
$$\tilde{\mu}_N = \frac{(A-1) + (N-B)}{N}\delta_{\ell_A/N} + \frac{1}{N} \sum_{i = A}^B \delta_{\ell_i/N}.$$
Note that $\tilde{\mu}_N \in \mathcal{M}(\Delta)$ and also 
\begin{equation*}
\begin{split}
d_1(\mu_N, \tilde{\mu}_N) \leq \sup_{F \subseteq \mathbb{R}} \left|\mu_N(F) - \tilde{\mu}_N(F) \right|  = \frac{(A-1) + (N-B)}{N} \leq \frac{A_1 + N - B_1}{N} < \epsilon,
\end{split}
\end{equation*}
where the supremum is over closed subsets $F$ of $\mathbb{R}$, and in the last inequality we used (\ref{PR2}). The last inequality implies that for $N \geq N_0$ we have 
$$\mathbb{P}_N^{\theta} \left(  d_1(\rho, \mu_N) < \epsilon \right) \leq \mathbb{P}_N^{\theta} \left(  d_1(\rho, \tilde{\mu}_N) < 2\epsilon \right) = 0,$$
where the last equality used that $\tilde{\mu}_N \in \mathcal{M}(\Delta)$ so that $ d_1(\rho, \tilde{\mu}_N) > 2\epsilon$. This establishes (\ref{PR1}).\\

%-------------------------------------------------------------------------------------------------------------------------------------------------------------------------------------------------
% Section 3.3
%
%-------------------------------------------------------------------------------------------------------------------------------------------------------------------------------------------------
\subsection{Proof of technical lemmas}\label{Section3.3} In this section we present the proofs of Lemmas \ref{ConvL} and \ref{CompactL}.

\begin{proof}[Proof of Lemma \ref{ConvL}] We first observe that 
\begin{equation*}
\begin{split}
\iint_{\mathbb{R}^2} {\bf 1}\{ x \neq y \} k_{V_N}(x,y) \mu^N(dx) \mu^N(dy) = &\iint_{[-A,A]^2}{\bf 1}\{ x \neq y \} \log |x -y|^{-1} \mu^N(dx) \mu^N(dy) \\
&+ \frac{N-1}{N} \int_{[-A,A]} V_N(x) \mu^N(dx).
\end{split}
\end{equation*}
Since $V_N$ converge to $V$ uniformly on $[-A,A]$ and $\mu^N$ converge weakly to $\mu^{\infty}$, we see that to prove (\ref{JK1}) it suffices to show that 
\begin{equation}\label{JK2}
\lim_{N \rightarrow \infty} \iint_{[-A,A]^2} {\bf 1}\{ x \neq y \} \log |x -y| \mu^N(dx) \mu^N(dy) = \iint_{[-A,A]^2} \log |x -y| \mu^\infty(dx) \mu^\infty(dy).
\end{equation}
For $M \in \mathbb{N}$ we let $f_M(x,y) = \max(\log |x -y|, - M)$ and then note that 
$$ \iint_{[-A,A]^2} {\bf 1}\{ x \neq y \} \log |x -y| \mu^N(dx) \mu^N(dy) \leq \iint_{[-A,A]^2} f_M(x,y) \mu^N(dx) \mu^N(dy) + \frac{M}{N},$$
which implies from the weak convergence of $\mu^N$ to $\mu^{\infty}$ and the continuity of $f_M$ that
$$\limsup_{N \rightarrow \infty} \iint_{[-A,A]^2} {\bf 1}\{ x \neq y \} \log |x -y| \mu^N(dx) \mu^N(dy)  \leq \iint_{[-A,A]^2} f_M(x,y) \mu^\infty(dx) \mu^\infty(dy).$$
Letting $M \rightarrow \infty$ in the last equation and using the dominated convergence theorem, with dominating function $\theta^{-2} {\bf 1}\{ (x,y) \in [-A,A]^2\} \cdot | \log |x-y||$, we obtain
\begin{equation}\label{JK3}
\limsup_{N \rightarrow \infty} \iint_{[-A,A]^2} {\bf 1}\{ x \neq y \} \log |x -y| \mu^N(dx) \mu^N(dy) \leq \iint_{[-A,A]^2} \log |x -y| \mu^\infty(dx) \mu^\infty(dy).
\end{equation}
We mention that in deriving (\ref{JK3}) we used that $\mu^{\infty} \in \mathcal{M}_{\theta}([-A,A])$.\\

In view of (\ref{JK3}), we see that to show (\ref{JK2}) it suffices to prove 
\begin{equation}\label{JK4}
\liminf_{N \rightarrow \infty} \iint_{[-A,A]^2} {\bf 1}\{ x \neq y \} \log |x -y| \mu^N(dx) \mu^N(dy) \geq \iint_{[-A,A]^2} \log |x -y| \mu^\infty(dx) \mu^\infty(dy).
\end{equation}
Let $\epsilon > 0$ be given, and for $M \in \mathbb{N}$ let $g_M(x,y)$ be a smooth function on $[-A, A]^2$, such that 
$$1 \geq g_M(x,y) \geq 0, \hspace{2mm} g_M(x,y) = 1 \mbox{ for } |x-y| \geq M^{-1}, \mbox{ and } g_M(x,y) = 0 \mbox{ for }|x-y| \leq (2M)^{-1}.$$
We note that for each $M, N \geq 1$
\begin{equation}\label{JK5}
\begin{split}
&\iint_{[-A,A]^2} \hspace{-3mm} {\bf 1}\{ x \neq y \} \log |x -y| \mu^N(dx) \mu^N(dy) \geq \iint_{[-A,A]^2} g_M(x,y) \log |x -y| \mu^N(dx) \mu^N(dy) \\
&+ N^{-2} \sum_{i = 1}^N \sum_{j =1, j \neq i}^N {\bf 1}\{ |\ell^N_i/N - \ell^N_j/N| \leq M^{-1}\} \log |\ell^N_i/N - \ell^N_j/N|  
\end{split}
\end{equation}
If we set $K_{M,N} = \lfloor N \theta^{-1} M^{-1} \rfloor$, we note that $N/M \geq K_{M,N} \theta$, $(K_{M,N} + 1)\theta > N/M$. The latter inequalities, combined with the statement $|\ell^N_i - \ell^N_j| \geq |i-j| \theta$ for $1 \leq i\neq j \leq N$, imply
\begin{equation}\label{JK6}
\begin{split}
&\sum_{i = 1}^N \sum_{j =1, j \neq i}^N {\bf 1}\{ |\ell^N_i/N - \ell^N_j/N| \leq M^{-1}\} \log |\ell^N_i/N - \ell^N_j/N|  \geq 2N \cdot \sum_{i = 1}^{K_{M,N}} \log \frac{i \theta}{N} \\
&= 2N K_{M,N}  \log \theta + 2N \log \left( K_{M,N}! \right) - 2N K_{M,N} \log N.
\end{split}
\end{equation}
From  \cite[Equation (1)]{Rob} we have
\begin{equation}\label{EqnRob}
n! = \sqrt{2\pi} n^{n+1/2} e^{-n} \cdot e^{r_n}, \mbox{ where } \frac{1}{12 n +1 } < r_n < \frac{1}{12n} \mbox{  for all $n \in \mathbb{N}$},
\end{equation}
and so if $N$ is sufficiently large, depending on $M,\theta$, we have
\begin{equation}\label{JK7}
\begin{split}
& 2N \log \left( K_{M,N}! \right) \geq 2N K_{M,N} \log K_{M,N} - 2N K_{M,N} +O(N \log N) \\
&= 2 N K_{M,N} \log N + \frac{2N^2}{\theta M} \log \left( \frac{1}{\theta M} \right) - 2N K_{M,N} + O(N \log N) , 
\end{split}
\end{equation}
where the constants in the big $O$ notations depend on $M, \theta$ and are possibly different. In deriving the equality in (\ref{JK7}) we used that $K_{M,N} = N \theta^{-1} M^{-1} + O(1)$.

Combining (\ref{JK5}), (\ref{JK6}) and (\ref{JK7}), we conclude that for each $M \in \mathbb{N}$
\begin{equation*}
\begin{split}
&\liminf_{N \rightarrow \infty} \iint_{[-A,A]^2} \hspace{-3mm} {\bf 1}\{ x \neq y \} \log |x -y| \mu^N(dx) \mu^N(dy) \\
&   \geq \liminf_{N \rightarrow \infty} \iint_{[-A,A]^2} g_M(x,y) \log |x -y| \mu^N(dx) \mu^N(dy) + \frac{2}{\theta M} \log \left( \frac{1}{\theta M} \right) + \frac{2 [\log \theta - 1]}{ \theta M} \\
& = \iint_{[-A,A]^2} g_M(x,y) \log |x -y| \mu^\infty(dx) \mu^\infty(dy)+ \frac{2}{\theta M} \log \left( \frac{1}{\theta M} \right) + \frac{2 [\log \theta - 1]}{ \theta M} ,
\end{split}
\end{equation*}
where in the last equality we used the weak convergence of $\mu^N$ to $\mu^{\infty}$ and the continuity of $g_M(x,y) \cdot \log |x-y|$. Taking $M \rightarrow \infty$ in the last line we get (\ref{JK4}) once we utilize the dominated convergence theorem with dominating function $\theta^{-2} {\bf 1}\{ (x,y) \in [-A,A]^2\} \cdot | \log |x-y||$.
\end{proof}

\begin{proof}[Proof of Lemma \ref{CompactL}] For clarity, we split the proof into two steps.  In the first step, we introduce some useful notation for our argument, and construct the measures $\mu_n$. In the second step, we show that the $\mu_n$ constructed in Step 1 satisfy the conditions of the lemma.\\

{\bf \raggedleft Step 1.} Observe that our assumption that $b > a$ implies that there are $c,d \in \mathbb{R}$, $c \leq d$ and $\epsilon_1 \in (0,1)$ such that $a \leq c - \epsilon_1$, $b \geq d + \epsilon_1$. We let $f(x)$ be the density of $\mu$, and since $\mu \in \mathcal{M}_{\theta}(\Delta)$ we may assume that $\theta^{-1} \geq f(x) \geq 0$ for all $x \in \mathbb{R}$.

For $\epsilon > 0$, we let $A_{\epsilon} = \{x \in [c - \epsilon_1, d + \theta + \epsilon_1]: f(x) > \theta^{-1} - \epsilon\}$ and $A^c_{\epsilon} = \{x \in [c - \epsilon_1, d + \theta + \epsilon_1 ]: f(x) \leq \theta^{-1} - \epsilon\}$. If $\lambda$ denotes the Lebesgue measure on $\mathbb{R}$, we see that for all $\epsilon \in (0, \theta^{-1})$ we have 
$$(\theta^{-1} - \epsilon) \lambda(A_{\epsilon}) \leq \mu(A_{\epsilon}) \leq 1,$$
which implies that there exists $\epsilon_2 \in (0, \theta^{-1})$ sufficiently small so that $\lambda ( A^c_{\epsilon_2}) \geq \epsilon_1$. We put $\epsilon = \min(\epsilon_1, \epsilon_2)$ and note that $\lambda ( A^c_{\epsilon}) \geq \epsilon$. 

Let $\{ a_n \}_{ n\in \mathbb{N}}$, $\{ a_n \}_{ n\in \mathbb{N}}$ be such that:
\begin{enumerate}
\item $a_{n+1} < a_n$ and $b_{n+1} > b_n$ for all $n \in \mathbb{N}$, $a_1 \leq c$, $b_1 \geq d$,
\item $\lim_{n \rightarrow \infty} a_n = a$, $\lim_{n \rightarrow \infty} b_n = b$.
\end{enumerate} 
We let $\rho_n = 1 - \mu([a_n, b_n])$ and observe that since $\mu \in \mathcal{M}_{\theta}(\Delta)$ (and thus has no atoms), we have $\lim_{n \rightarrow \infty} \rho_n = 0$. The latter implies that there exists $N_0 \in \mathbb{N}$ such that for $n \geq N_0$ we have $\rho_n \leq \epsilon^2$. 

Let us define the functions $g_n$ through 
\begin{equation}\label{S3DefGN}
g_n(x) = \begin{cases} f(x) \cdot {\bf 1}\{ x \in [a_n, b_n]\} + \frac{\rho_n}{\lambda( A^c_{\epsilon})}  \cdot {\bf 1} \{ x\in A_\epsilon^c\} &\mbox{ if } n \geq N_0, \\ \theta^{-1} \cdot {\bf 1}\{ x \in [c, c+ \theta]\} & \mbox{ if } n < N_0. \end{cases}
\end{equation}
It is clear from the definition of $\rho_n$ that $g_n(x)$ are probability density functions on $\mathbb{R}$, and we let $\mu_n$ be the corresponding measures.\\

{\bf \raggedleft Step 2.} We proceed to prove that $\mu_n$ satisfy the conditions of the lemma. 

We first check that $\mu_n \in \mathcal{M}_{\theta}(\mathbb{R})$. The latter is clear if $n < N_0$, so we assume that $n \geq N_0$. From (\ref{S3DefGN}), we see that  
$$ g_n(x) \leq f(x) \leq \theta^{-1} \mbox{ if $x \not \in A_{\epsilon}^c$, and } g_n(x) \leq (\theta^{-1} - \epsilon) + \frac{\rho_n}{\lambda( A^c_{\epsilon})} \leq \theta^{-1},$$
where we used that $\rho_n \leq \epsilon^2$ and $\lambda( A^c_{\epsilon}) \geq \epsilon$. Thus $\mu_n \in \mathcal{M}_{\theta}(\mathbb{R})$.

By construction, we know that $\mu_n$ are supported on $[a_n, b_n]$, and $\mu_n$ weakly converge to $\mu$ (say by the Portmanteau theorem). Thus we only need to show that 
\begin{equation*}
\limsup_{n \rightarrow \infty} E_V(\mu_n) \leq E_V(\mu) \mbox{ and }\liminf_{n \rightarrow \infty} E_V(\mu_n) \geq E_V(\mu),
\end{equation*}
which in view of (\ref{S2EFEq}) is equivalent to
\begin{equation}\label{LBE}
\limsup_{n \rightarrow \infty} E_{\mathcal{V}}(\nu_n) \leq E_{\mathcal{V}}(\nu) \mbox{ and }\liminf_{n \rightarrow \infty} E_{\mathcal{V}}(\nu_n) \geq E_{\mathcal{V}}(\nu),
\end{equation}
where $\nu_n = T_* \mu_n$, $\nu  = T_*\mu$ as in (\ref{S2Push}). In the remainder we focus on proving (\ref{LBE}).\\

As $F_{\mathcal{V}}(\vec{x}, \vec{y})$ as in (\ref{S2Ker}) is lower semi-continuous on $\mathcal{S} \times \mathcal{S}$, we know that there exists an increasing sequence of continuous functions $ F^M_{\mathcal{V}}(\vec{x}, \vec{y})$ that converge pointwise to $F_{\mathcal{V}}(\vec{x}, \vec{y})$ as $M \rightarrow \infty$. The latter shows that 
\begin{equation*}
\liminf_{n \rightarrow \infty} E_{\mathcal{V}}(\nu_n) \geq \liminf_{n \rightarrow \infty} \iint_{\mathcal{S}^2} F^M_{\mathcal{V}}(\vec{x}, \vec{y}) \nu_n(d\vec{x}) \nu_n(d\vec{y}) = \iint_{\mathcal{S}^2} F^M_{\mathcal{V}}(\vec{x}, \vec{y}) \nu(d\vec{x}) \nu(d\vec{y}),
\end{equation*}
where we used the continuity of $F^M_{\mathcal{V}}$ and the fact that $\nu_n$ converge weakly to $\nu$ (this follows from the weak convergence of $\mu_n$ to $\mu$ and Lemma \ref{S2TDiff}). By the monotone convergence theorem, the right side above converges to $ E_{\mathcal{V}}(\nu)$ as $M\rightarrow \infty$, which proves the second inequality of (\ref{LBE}).\\

 Let us write $\mu_n^1$ to be the probability measure with density $(1- \rho_n)^{-1} f(x) \cdot {\bf 1}\{ x \in [a_n, b_n]\} $ and $\mu^2$ the one with density $[\lambda( A^c_{\epsilon})]^{-1} \cdot {\bf 1} \{ x\in A_\epsilon^c\}$ for $n \geq N_0$. We also let $\nu_{n}^1 = T_* \mu_{n,1}$ and $\nu^2 = T_* \mu^2$. Then, we have for $n \geq N_0$  
$$\mu_n = (1- \rho_n) \cdot \mu_n^1 + \rho_n \cdot \mu^2\mbox{ and }\nu_n = (1- \rho_n) \cdot \nu_n^1 + \rho_n \cdot \nu^2.$$
Using the convexity of $E_{\mathcal{V}}$ on $\mathcal{M}(\mathcal{S})$, see Proposition \ref{S2GRF}, we know that 
$$E_{\mathcal{V}}(\nu_n)  \leq (1 - \rho_n)  \cdot E_{\mathcal{V}}(\nu^1_n) + \rho_n \cdot E_{\mathcal{V}}(\nu^2).$$
Notice that 
$$ \left|  E_{\mathcal{V}}(\nu^2) \right| = \left|  E_{V}(\mu^2) \right| \leq \epsilon^{-2} \cdot \iint_{[c - \epsilon_1, d + \theta + \epsilon_1]^2} |\log|x-y|| dx dy + \sup_{x \in [c - \epsilon_1, d + \theta + \epsilon_1]} |V(x)| < \infty,$$
which implies that $\lim_{n \rightarrow \infty} \rho_n \cdot E_{\mathcal{V}}(\nu^2) = 0$ as $\lim_{n \rightarrow \infty} \rho_n = 0$.

Combining the last few statements, we see that to prove the first inequality in (\ref{LBE}) it suffices to show that
\begin{equation}\label{FDE2}
\begin{split}
&\lim_{n \rightarrow \infty} (1 - \rho_n)  \cdot E_{\mathcal{V}}(\nu^1_n) =  E_{\mathcal{V}}(\nu) \iff \lim_{n \rightarrow \infty} (1 - \rho_n)^2  \cdot E_{\mathcal{V}}(\nu^1_n) =  E_{\mathcal{V}}(\nu)  \iff \\
& \lim_{n \rightarrow \infty} \iint_{\mathcal{S}^2}{\bf 1}\{\vec{x}, \vec{y}\in T([a_n, b_n])\} \cdot F_{\mathcal{V}}(\vec{x}, \vec{y}) \nu(d\vec{x})\nu(d\vec{y}) =  \iint_{\mathcal{S}^2} F_{\mathcal{V}}(\vec{x}, \vec{y}) \nu(d\vec{x})\nu(d\vec{y}).
\end{split}
\end{equation} 
Notice that $T([a_n, b_n])$ form an increasing sequence of sets and since $\nu(T([a_n, b_n])) = 1 - \rho_n$, we have $\nu \left( \cup_{n \geq 1} T([a_n, b_n]) \right) = 1$. The latter and the lower boundedness of $F_{\mathcal{V}}(\vec{x}, \vec{y})$ on $\mathcal{S}\times \mathcal{S}$ allows us to conclude the second line in (\ref{FDE2}) from the monotone convergence theorem.
\end{proof}

%-------------------------------------------------------------------------------------------------------------------------------------------------------------------------------------------------
% Section 4
%
%-------------------------------------------------------------------------------------------------------------------------------------------------------------------------------------------------
\section{Applications}\label{Section4} In this section, we give two brief applications of Theorem \ref{ThmMain}. In Section \ref{Section4.1}, we consider certain measures related to Jack symmetric functions, and in Section \ref{Section4.2}, we consider discrete analogues of the Cauchy ensembles from \cite[Example 1.3]{Hardy}. We continue with the notation from Section \ref{Section1.2}.

%-------------------------------------------------------------------------------------------------------------------------------------------------------------------------------------------------
% Section 4.1
%
%-------------------------------------------------------------------------------------------------------------------------------------------------------------------------------------------------
\subsection{Jack measures}\label{Section4.1} Fix $\theta, t \in (0, \infty)$, $N \in \mathbb{N}$. Let  $\P^{\operatorname{Jack}}_{N }$ be the measure on $\mathbb{W}^{\theta}_{N}(0, \infty) $, given by
\begin{align}\label{jackM2}
\P^{\operatorname{Jack}}_{N } (\ell_1, \dots, \ell_N)= \frac{1}{Z_N} \cdot \prod_{1\le i<j\le N} Q_{\theta}(\ell_i - \ell_j) \prod_{i=1}^N e^{-  \theta N V_N(\ell_i/N)},
\end{align}
where $Q_{\theta}$ is as in (\ref{PDef}),
\begin{equation}\label{jackM3}
V_N(x) = \frac{1}{\theta N}\log\frac{\Gamma(Nx+1 )}{(t\theta N)^{Nx}}, \mbox{ and }Z_N=  \Gamma(\theta)^{-N}e^{t\theta N^2}(t \theta N)^{\frac{N(N-1)}2} \cdot \prod_{i=1}^N\Gamma(i\theta).
\end{equation}
The measure $\P^{\operatorname{Jack}}_{N }$ arises as a special case of the {\em Jack measures}, which are probability measures on partitions related to Jack symmetric functions, and in turn are special cases of the Macdonald measures from \cite{bor-cor}. We refer the interested reader to \cite[Section 6.3]{DD21}, where the relationship to Jack symmetric functions is explained in detail and it is shown that $\P^{\operatorname{Jack}}_{N }$ is a well-defined probability measure on $\mathbb{W}^{\theta}_{N}(0, \infty) $. We also mention here that the measure $\P^{\operatorname{Jack}}_{N } $ was previously studied in \cite{misha}, where it arises as the time $tN$ distribution of a certain Markov process on partitions, which is a discrete version of $\beta$-Dyson Brownian motion (here $\beta = 2 \theta$).

We have the following result about the measures $\P^{\operatorname{Jack}}_{N }$.
\begin{corollary}\label{S4LDP1}Fix $\theta, t \in (0, \infty)$, $N \in \mathbb{N}$ and let $\P^{\operatorname{Jack}}_{N }$ be as in (\ref{jackM2}). Let $\mu_N = N^{-1}\sum_{i = 1}^N \delta_{\ell_i/N}$ be the empirical measures of $ (\ell_1, \dots, \ell_N)$, distributed according to $\P^{\operatorname{Jack}}_{N }$. Then, the sequence of measures in $\mathcal{M}(\mathbb{R})$, given by the laws of $\mu_N$, satisfies an LDP with speed $N^2$ and good rate function
\begin{equation}\label{IVJack}
I_V^{\operatorname{Jack}}(\mu):= 
\begin{cases} \theta ( E_V(\mu)  - E_V(\me^{\operatorname{Jack}}) ) &\mbox{ for $\mu \in \mathcal{M}_{\theta}([0,\infty))$} \\
 \infty & \mbox{ for $\mu \in \mathcal{M}(\mathbb{R}) \setminus  \mathcal{M}_{\theta}([0,\infty))$ }
 \end{cases}
\end{equation}
where $E_V$ is as in (\ref{S1EF}) and $V(x) = \theta^{-1} \left(x\log x - \log(e t \theta) x \right)$.  Here $\me^{\operatorname{Jack}}$ is a probability measure on $[0, \infty)$ with density $\phi^{\operatorname{Jack}}$, which for $t\ge 1$ is equal to
\begin{align*}
\phi^{\operatorname{Jack}}(x)=\begin{cases} (\theta\pi)^{-1}\operatorname{arccot}\left(\dfrac{x+\theta(t-1)}{\sqrt{4\theta tx-[x+\theta(t-1)]^2}}\right) &  \mbox{for } x\in(\theta(\sqrt{t}-1)^2,\theta(\sqrt{t}+1)^2), \\ 0 & \mbox{otherwise.}
\end{cases}
\end{align*}
and for $t\in (0,1)$ is given by
\begin{align*}
\phi^{\operatorname{Jack}}(x)=\begin{cases}
(\theta\pi)^{-1}\operatorname{arccot}\left(\dfrac{x+\theta(t-1)}{\sqrt{4\theta tx-[x+\theta(t-1)]^2}}\right) &  \mbox{for } x\in(\theta(\sqrt{t}-1)^2,\theta(\sqrt{t}+1)^2),
\\ \theta^{-1} & \mbox{for }0 \leq x<\theta(\sqrt{t}-1)^2,  \\ 0 & \mbox{for }  x>\theta(\sqrt{t}+1)^2. 
\end{cases}
\end{align*} 
\end{corollary}
\begin{proof} We first observe that $\P^{\operatorname{Jack}}_{N }$ is of the form (\ref{PDef}) with $a_N = 0$ and $b_N = \infty$ for each $N \in \mathbb{N}$ so that Assumption 1 from Section \ref{Section1.2} is satisfied with $\Delta = [0, \infty)$.

As explained in \cite[Section 6.3]{DD21}, the functions $V_N$ are continuous on $[0, \infty)$ and there exists a constant $A > 0$, depending on $\theta, t$, such that for all $N \geq 1$ and $x \geq 0$ we have 
$$A + V_N(x) \geq 2 \log (1 + x^2).$$
In addition, in \cite[Section 6.3]{DD21} it is shown that for each $n \in \mathbb{N}$ there is a constant $A_n > 0$ (depending on $n$, $\theta$ and $t$) such that for all $N \geq 1$
$$\sup_{x \in [0, n]} |V_N(x) - V(x)| \leq A_n N^{-1}\log (N+1).$$
The latter observations show that $V_N$ and $V$ satisfy the conditions in Assumption 2(b) (here we extend $V$ and $V_N$ to $\mathbb{R}$ by setting $V(-x) = V(x)$ and $V_N(-x) = V_N(x)$ for $x \geq 0$). We conclude from Theorem \ref{ThmMain} that the sequence of measures in $\mathcal{M}(\mathbb{R})$, given by the laws of $\mu_N$, satisfies an LDP with speed $N^2$ and good rate function $I_V^{\theta}$ as in that theorem.

What remains is to show that $I_V^{\theta} = I_V^{\operatorname{Jack}}$ as in (\ref{IVJack}). From \cite[Lemma 6.11]{DD21} we have that $\me^{\operatorname{Jack}}$ is the unique minimizer of $E_V$ on $\mathcal{M}_{\theta}([0,\infty))$, which in view of Theorem \ref{ThmFunct} shows that $I_V^{\theta} = I_V^{\operatorname{Jack}}$.
\end{proof}

%-------------------------------------------------------------------------------------------------------------------------------------------------------------------------------------------------
% Section 4.2
%
%-------------------------------------------------------------------------------------------------------------------------------------------------------------------------------------------------
\subsection{Discrete Cauchy ensembles}\label{Section4.2} Fix $\theta \in (1/2, \infty)$, $N \in \mathbb{N}$. Let  $\P^{\operatorname{Cauchy}}_{N}$ be the measure on $\mathbb{W}^{\theta}_{N}(-\infty, \infty) $, given by
\begin{align}\label{CauchyM2}
\P^{\operatorname{Cauchy}}_{N } (\ell_1, \dots, \ell_N)= \frac{1}{Z_N} \cdot \prod_{1\le i<j\le N} Q_{\theta}(\ell_i - \ell_j) \prod_{i=1}^N e^{-  \theta N V(\ell_i/N)},
\end{align}
where $Q_{\theta}$ is as in (\ref{PDef}), $V(x) = \log (1 + x^2)$ and $Z_N$ is a normalization constant. We mention that $Z_N$ is finite and $\P^{\operatorname{Cauchy}}_{N }$ a well-defined probability measure on $\mathbb{W}^{\theta}_{N}(-\infty, \infty)$ from Lemma \ref{WellDef}. Here it is important that $\theta > 1/2$ and if $\theta \in (0, 1/2]$ then (\ref{CauchyM2}) is not a well-defined probability measure since for $W(\ell) = \prod_{1 \leq i < j \leq N} Q_{\theta}(\ell_i-\ell_j)  \prod_{i = 1}^N e^{- \theta N V(\ell_i/N)}$ we have
\begin{equation*}
\begin{split}
&\sum_{{\ell} \in \mathbb{W}^{\theta}_{N}(-\infty, \infty) } W({\ell}) \geq \sum_{\substack{ {\ell} \in \mathbb{W}^{\theta}_{N}(-\infty, \infty): \\ \ell_i = (N-i)\theta \mbox{ for }i = 2, \dots, N }} W({\ell}) = C_1\sum_{n = 0}^{\infty} \prod_{j = 2}^N  Q_{\theta}(n + (j-1)\theta  )   \\
&\times e^{-  \theta N \log (1 + [n + (N-1) \theta]^2/N^2)} \geq C_2\sum_{n = 0}^{\infty} e^{ \sum_{j = 2}^N 2\theta \log (n + (j-1)\theta ) -  \theta N \log (1 + [n + (N-1) \theta]^2/N^2)} = \infty.
\end{split}
\end{equation*}
In the last set of inequalities we have that $C_1, C_2$ are positive constants, depending on $\theta$ and $N$. The first inequality on the second line follows from Lemma \ref{InterApprox} and the last equality follows from the fact that if $c_n$ is the $n$-th summand we have $c_n \sim n^{-2\theta}$ as $n \rightarrow \infty$ and $\theta \in (0, 1/2]$. 

When $\theta = 1$ we recall that $\ell_i \in \mathbb{Z}$ for all $i \in \{1, \dots, N\}$ and $Q_{\theta}(x) = x^2$, in which case we observe that $\P^{\operatorname{Cauchy}}_{N }$ from (\ref{CauchyM2}) is a discrete analogue of the {\em Cauchy ensemble} from \cite[Example 1.3]{Hardy}.

We have the following result about the measures $\P^{\operatorname{Cauchy}}_{N }$.
\begin{corollary}\label{S4LDP2}Fix $\theta \in (1/2, \pi]$, $N \in \mathbb{N}$ and let $\P^{\operatorname{Cauchy}}_{N}$ be as in (\ref{CauchyM2}). Let $\mu_N = N^{-1}\sum_{i = 1}^N \delta_{\ell_i/N}$ be the empirical measures of $(\ell_1, \dots, \ell_N)$, distributed according to $\P^{\operatorname{Cauchy}}_{N} $. Then, the sequence of measures in $\mathcal{M}(\mathbb{R})$, given by the laws of $\mu_N$, satisfies an LDP with speed $N^2$ and good rate function
\begin{equation}\label{IVCauchy}
I_V^{\operatorname{Jack}}(\mu):= 
\begin{cases} \theta ( E_V(\mu)  - E_V(\me^{\operatorname{Cauchy}}) ) &\mbox{ for $\mu \in \mathcal{M}_{\theta}(\mathbb{R})$} \\
 \infty & \mbox{ for $\mu \in \mathcal{M}(\mathbb{R}) \setminus  \mathcal{M}_{\theta}(\mathbb{R})$ }
 \end{cases}
\end{equation}
where $E_V$ is as in (\ref{S1EF}). Here $\me^{\operatorname{Cauchy}}$ is the Cauchy distribution on $\mathbb{R}$, i.e. the one with density 
\begin{align*}
\phi^{\operatorname{Cauchy}}(x)= \frac{1}{\pi(1 + x^2)}. 
\end{align*}
\end{corollary}
\begin{remark} Let us explain the restriction of the parameter $\theta$ in Corollary \ref{S4LDP2}. As we explained in the beginning of the section, we require that $\theta > 1/2$ so that $\P^{\operatorname{Cauchy}}_{N}$ is well-defined. The requirement that $\theta \leq \pi$ is imposed so that the minimizer of $ E_V$ is precisely the Cauchy distribution. In general, Theorem \ref{ThmMain} is applicable to $ \P^{\operatorname{Cauchy}}_{N}$ for any $\theta > 1/2$ and implies that the laws of $\mu_N$, satisfies an LDP with speed $N^2$ and good rate function $I_V^{\theta}$ as in Theorem \ref{ThmFunct}. If $\me^{\theta}$ is as in Theorem \ref{ThmFunct} for $\Delta = \mathbb{R}$ and $V = \log (1 + x^2)$ we will see in the proof of Corollary \ref{S4LDP2} that $\me^{\theta} = \me^{\operatorname{Cauchy}}$, provided that $\theta \in (0, \pi]$. If $\theta > \pi$ then $\me^{\operatorname{Cauchy}} \not \in \mathcal{M}_{\theta}(\mathbb{R})$ and so we necessarily have $\me^{\theta} \neq \me^{\operatorname{Cauchy}}$.

 It would be interesting to also find a formula for $\me^{\theta}$ when $\theta > \pi$, and one possible way to approach this question is to first {\em guess} a formula $\me^{\theta}$ using ideas that are similar to those in \cite[Section 4]{ds97}, \cite[Section 5]{fe} and \cite[Section 6]{jo}. Once a formula for $\me^{\theta}$ is obtained, one can use the variational characterization of the equilibrium measure, see part (4) of Theorem \ref{ThmFunct}, to verify that it is indeed the correct one. 

We will not pursue the formula for $\me^{\theta}$ when $\theta > \pi$ and refer the interested reader to  \cite[Section 4]{ds97}, \cite[Section 5]{fe}, \cite[Section 6]{jo} and more recently \cite[Section 6]{DD21}, for related contexts where the approach we described above has been carried out.
\end{remark}

\begin{proof} We first observe that $\P^{\operatorname{Cauchy}}_{N }$ is of the form (\ref{PDef}) with $a_N = -\infty$ and $b_N = \infty$ for each $N \in \mathbb{N}$ so that Assumption 1 from Section \ref{Section1.2} is satisfied with $\Delta = \mathbb{R}$. It is also clear that $V(x)$ satisfies the conditions of Assumption 2(a) with $\theta_N' = \theta$ in (\ref{VNgrowth}). We conclude from Theorem \ref{ThmMain} that the sequence of measures in $\mathcal{M}(\mathbb{R})$, given by the laws of $\mu_N$, satisfies an LDP with speed $N^2$ and good rate function $I_V^{\theta}$ as in that theorem. This is true for any $\theta > 1/2$.

What remains is to show that $I_V^{\theta} = I_V^{\operatorname{Cauchy}}$ as in (\ref{IVCauchy}) when $\theta \in (1/2,\pi]$. In \cite[Example 1.3 and Remark 2.2]{Hardy} it was shown using an elegant symmetry argument that $\me^{\operatorname{Cauchy}}$ is the unique minimizer of $E_V$ over $\mathcal{M}(\mathbb{R})$ and since $\me^{\operatorname{Cauchy}} \in \mathcal{M}_{\theta}(\mathbb{R})$ when $\theta \in (0, \pi]$ we conclude that $ \me^{\theta}  = \me^{\operatorname{Cauchy}}$ from part (2) of Theorem \ref{ThmFunct}. This proves that $I_V^{\theta} = I_V^{\operatorname{Cauchy}}$ when $\theta \in (1/2,\pi]$.
\end{proof}
\begin{remark}
In the last part of the above proof, one can also deduce that $ \me^{\theta}  = \me^{\operatorname{Cauchy}}$ when $\theta \in (1/2,\pi]$ from part (4) of Theorem \ref{ThmFunct}. Indeed, by a direct computation for all $y \in \mathbb{R}$ 
\begin{equation}\label{S4ER1}
\int_\R \left(\log|x-y|^{-1}+\frac{1}{2}\log( 1+x^2) \right)\me^{\operatorname{Cauchy}}(dx) + \frac{1}{2} \log(1 + y^2) = \int_\R \frac{\log( 1+x^2) dx}{2\pi (1 + x^2)} \in (0,\infty).
\end{equation}
One also has from (\ref{S4ER1}) that  
\begin{equation*}
\begin{split}
E_V(\me^{\operatorname{Cauchy}}) = \hspace{2mm}& \int_{\mathbb{R}}\int_\R \left(\log|x-y|^{-1}+\frac{1}{2}\log( 1+x^2) + \frac{1}{2}\log( 1+y^2)  \right)\me^{\operatorname{Cauchy}}(dx) \me^{\operatorname{Cauchy}}(dy) \\
& = \int_\R \int_\R \frac{\log( 1+x^2) dx}{2\pi (1 + x^2)} \cdot  \frac{dy}{\pi (1 + y^2)} = \int_\R \frac{\log( 1+x^2) dx}{2\pi (1 + x^2)}  
\end{split}
\end{equation*}
The fact that the integral in (\ref{S4ER1}) does not depend on $y \in \mathbb{R}$, $E_V(\me^{\operatorname{Cauchy}}) < \infty$, and $ \me^{\operatorname{Cauchy}} \in \mathcal{M}(\mathbb{R})$ together imply $\me^{\theta}  = \me^{\operatorname{Cauchy}}$ in view of part (4) of Theorem \ref{ThmFunct}.
\end{remark}

\bibliographystyle{alphaabbr} 
\bibliography{PD}

\end{document}